\documentclass[a4paper,11pt]{article}

\usepackage{amsmath}
\usepackage{amsfonts}
\usepackage{amssymb}
\usepackage{amsthm}
\usepackage{graphicx}
\usepackage{psfrag}
\usepackage{subfigure}
\usepackage{xcolor}

\newcommand{\E}{\ensuremath{\mathcal{E}}}
\newcommand{\p}{\varphi}
\newcommand{\R}{\mathbb{R}}
\newcommand{\Rn}{\mathbb{R}^d}
\newcommand{\weight}[1]{\langle #1\rangle}

\newcommand{\ZuWeis}{\mathrel{\mathop:\!\!=}}

\newcommand{\D}{\mathcal{D}}
\newcommand{\La}{\mathcal{L}}
\newcommand{\F}{\mathcal{F}}
\newcommand{\K}{\mathcal{K}}
\newcommand{\bv}{\mathbf{v}}
\newcommand{\bw}{\mathbf{w}}
\newcommand{\boldf}{\mathbf{f}}
\newcommand{\tbj}{\widetilde{\mathbf{J}}}
\newcommand{\bpsi}{\boldsymbol{\psi}}

\newcommand{\loc}{\operatorname{loc}}
\newcommand{\uloc}{\operatorname{uloc}}
\newcommand{\N}{\ensuremath{\mathbb{N}}}
\newcommand{\Hone}{H^1_{(0)}}
\newcommand{\supp}{\operatorname{supp}}
\newcommand{\ol}[1]{\overline{#1}}
\newcommand{\sd}{\, d}
\newcommand{\pv}{\operatorname{p.v.}}
\newcommand{\eps}{\ensuremath{\varepsilon}}

\newcommand{\Hs}[1][1]{\ensuremath{H^{\protect #1}}}

\newcommand{\Lp}[1][2]{\ensuremath{L^{#1}}}
\newcommand*{\norm}[2]{\| #2 \|_{#1}}

\newcommand{\LL}{\mathcal{L}}
\newcommand{\dom}{\operatorname{dom}}

\newcommand{\Ltwoprod}[2]{\left( #1, #2 \right)_{L^2}}
\newcommand{\duality}[2]{\left\langle #1, #2 \right\rangle}
\newcommand{\longduality}[4]{\left\langle #2, #3 \right\rangle_{#1, #4}}

\newcommand{\no}{\mathbf{n}}

\newtheorem{definition}{Definition}[section]
\newtheorem{remark}[definition]{Remark}
\newtheorem{lemma}[definition]{Lemma}
\newtheorem{theorem}[definition]{Theorem}
\newtheorem{assumption}[definition]{Assumption}

\numberwithin{equation}{section}  %equation-counter is set back to zero at the beginning of each section

\newcommand{\changes}[1]{{ #1}}

\begin{document}
\begin{titlepage}
\title{Weak Solutions for a Diffuse Interface Model for Two-Phase Flows of 
Incompressible Fluids with Different Densities and Nonlocal Free Energies}
\author{  Helmut Abels\footnote{Fakult\"at f\"ur Mathematik,  
Universit\"at Regensburg,
93040 Regensburg,
Germany, e-mail: {\sf helmut.abels@mathematik.uni-regensburg.de}}\ \ and 
Yutaka Terasawa\footnote{Graduate School of Mathematics, Nagoya University, Furocho Chikusaku, Nagoya, 464-8602, Japan, e-mail: {\sf yutaka@math.nagoya-u.ac.jp}}
}
\date{}
\end{titlepage}
\maketitle
\begin{abstract}
  We consider a diffuse interface model for the flow of two viscous
 incompressible Newtonian fluids with different densities in a bounded domain in two and three space dimensions and prove existence of weak solutions for it. In contrast to earlier contributions, we study a model with a singular non-local free energy, which controls the $H^{\alpha/2}$-norm of the volume fraction. We show existence of weak solutions for large times with the aid of an implicit time discretization. 
\end{abstract}

\noindent{\bf AMS-Classification:} 
Primary: 76T99; %% Two-Phase flows: Others
Secondary: 35Q30, %% Stokes and Navier-Stokes eq.
35Q35, %% Other equations arising in fluid mechanics
76D03, %% Incompressible viscous fluids: Existence, uniqueness, and regularity theory 
76D05, %% Incompressible viscous fluids: Navier-Stokes equations
76D27, %% Incompressible viscous fluids: Other free-boundary flows; Hele-Shaw flows
76D45 %% Incompressible viscous fluids: Capillarity (surface tension)

%%%%%%%%%%%%%%%%%%%%%%%%%%%%%%%%%%%%%%%%%%%%%%%%%%%%%%%%%%%%%%%%%%%%%%%%%%%%%%%%%%%%%%%%%%%%%%%%
%%%%%%%%%%%%%%%%%%%%%%%%%%%%%%%%%%%%%%%%%%%%%%%%%%%%%%%%%%%%%%%%%%%%%%%%%%%%%%%%%%%%%%%%%%%%%%%%
%%%%%%%%%%%%%%%%%%%%%%%%%%%%%%%%%%%%%%%%%%%%%%%%%%%%%%%%%%%%%%%%%%%%%%%%%%%%%%%%%%%%%%%%%%%%%%%%

\section{Introduction} \label{intro}

%%% Attention: The following part is as in first ADG-paper!
In this contribution, we consider a two-phase flow for incompressible fluids of different densities and different viscosities. The two fluids are assumed to be macroscopically immiscible and to be miscible in a thin interface region, i.e., we consider a diffuse interface model
(also called phase field model) for the two-phase flow. 
In contrast to sharp interface models, where the interface between the two fluids is a sufficiently smooth hypersurface, diffuse interface model can describe topological changes due to pinch off and droplet collision. 

There are several diffuse interface models for such two-phase flows. Firstly, in the case of matched densities, i.e.,  the densities of both fluids are assumed to be identical, there is a well-known model H, cf.\ Hohenberg and Halperin or Gurtin et al. \cite{HH77, GPV96}.
In the case that the fluid densities do not coincide there are different models. On one hand  Lowengrub and Truskinovsky \cite{LT98} derived
a quasi-incompressible model, where the  mean velocity field of the mixture is in general not
divergence free. On the other hand, Ding et al. \cite{DSS07} proposed a model
with a divergence free mean fluid velocities. But this model is not known to be thermodynamically 
consistent. In \changes{Abels}, Garcke and Gr\"un \cite{AGG11}
a thermodynamically consistent diffuse interface model for two-phase flow
with different densities and a divergence free mean velocity field was derived, which we call AGG model for short. 
The existence of weak solutions of the AGG model was shown in \cite{ADG13}. For analytic result in the case of matched densities, i.e., the model H, we refer to \cite{Abe09b} \changes{and \cite{GMT19}} and the reference given there. Existence of weak and strong solutions for \changes{a slight modification} of the model by Lowengrub and Truskinovsky was proven in \cite{Abe09a,LTModelShortTime}. 

Concerning the Cahn-Hilliard equation, Giacomin and Lebowitz \cite{GL97, GL98} observed that a physically more rigorous derivation leads to a nonlocal equation, which we call a nonlocal Cahn-Hilliard equation. 
There are two types of nonlocal Cahn-Hilliard equations. One is the equation where the second order differential operator in the equation for the chemical potential is replaced by a convolution operator with a sufficiently smooth even function. We call it a nonlocal Cahn-Hilliard equation with a regular kernel in the following. 
The other is one where the second order differential operator is replaced by a regional fractional Laplacian. We call it a nonlocal Cahn-Hilliard equation with a singular kernel, since the regional fractional Laplacian is defined by using singular kernel. The nonlocal Cahn-Hilliard equation with a regular kernel was analyzed in \cite{GZ03, GG14, GL98, LP11a, LP11b}. On the other hand, the nonlocal Cahn-Hilliard equation with a singular kernel was first analyzed in Abels, Bosia and Grasselli \cite{ABG15}, where they proved the existence and uniqueness of a weak solution of the nonlocal Cahn-Hilliard equation, its regularity properties and the existence of a (connected) global attractor. %Concerning corresponding fractional Allen-Cahn equation, there is an earlier study by \cite{NNG08}.

Concerning the nonlocal model H with a regular kernel, where the convective Cahn-Hilliard equation is replaced by the convective nonlocal Cahn-Hilliard equation with a regular kernel, first studies were done by \cite{CFG12, FG12a, FG12b}\changes{, see also \cite{FGGS19} and the references there for more recent results}. More recently, the nonlocal AGG model with a regular kernel,
 where the convective Cahn-Hilliard equation is replaced by the convective nonlocal Cahn-Hilliard equation with a regular kernel, was studied by Frigeri~\cite{F15} and he showed the existence of a weak solution for that model.
The method of the proof in \cite{F15} is based on the Faedo-Galerkin method of a suitably mollified system and the method of passing to the limit with two parameters tending to zero. The method is different from \cite{ADG13} which is based on implicit time discretization and a Leray-Schauder fixed point argument. 

In this contribution, we consider a nonlocal AGG model with a singular kernel, where a convective Cahn-Hilliard equation in the AGG model is replaced by a convective nonlocal Cahn-Hilliard equation with a singular kernel. Our aim is to prove the existence of a weak solution of such a system. 

In this contribution we consider existence of weak solutions of the following system, which couples a nonhomogeneous Navier-Stokes equation system with a nonlocal Cahn-Hilliard equation:
\begin{align}
 \partial_t (\rho \mathbf{v}) + \operatorname{div} ( \bv \otimes(\rho \bv + \tbj)) - \operatorname{div} (2 \eta(\varphi) D \bv)
  + \nabla p  
  & = \mu \nabla \varphi & \mbox{in } \, Q ,  \label{eq:1} 
\\
 \operatorname{div} \, \bv &= 0& \mbox{in } \, Q,  \label{eq:2} 
\\
 \partial_t \varphi + \bv \cdot \nabla \varphi &= \mbox{div}\left(m(\varphi) \nabla \mu \right)& \mbox{in } \, Q, \label{eq:3} 
\\
 \mu = \Psi'(\varphi)  &+\LL\varphi & \mbox{in } \, Q , \label{eq:4} 
\end{align}
where $\rho=\rho(\varphi):= \frac{\tilde{\rho}_1+\tilde{\rho}_2}2+ \frac{\tilde{\rho}_2-\tilde{\rho}_1}2\varphi $, $\tbj = -\frac{\tilde{\rho}_2 - \tilde{\rho}_1}{2} m(\varphi) \nabla \mu$, 
$Q=\Omega\times(0,\infty)$. We assume that $\Omega \subset \mathbb{R}^d$, $d=2,3$, is a bounded domain with $C^2$-boundary. Here and in the following $\mathbf{v}$, $p$, and $\rho$ are the (mean) velocity, the pressure and the
density of the mixture of the two fluids, respectively. Furthermore $\tilde{\rho}_j$, $j=1,2$, are the specific densities of the unmixed fluids, 
 $\varphi$ is the difference of the volume fractions of the two fluids, and $\mu$ is the chemical
potential related to $\varphi$. 
Moreover,  ${D}\mathbf{v}= \frac12(\nabla \mathbf{v} + \nabla \mathbf{v}^T)$,
$\eta(\varphi)>0$ is the viscosity of the fluid mixture, and $m(\varphi)>0$ is a mobility coefficient. \changes{The term $\tbj$ describes the mass flux, i.e., we have
\begin{equation*}
  \partial_t \rho = - \operatorname{div} \tbj. 
\end{equation*}
It is important to have the term with $\tbj$ in \eqref{eq:1} in order to obtain a thermodynamically consistent model, cf.\ \cite{AGG11} for the case with a local free energy.}

Finally,
 $\LL$ is defined as 
\begin{align}\label{eq:defnL}
    \LL u(x) &= \pv \int_{\Omega} (u(x)-u(y))k(x,y,x-y)dy\\\nonumber
    &=\lim_{\eps\to 0} \int_{\Omega\setminus B_\eps(x)} (u(x)-u(y))k(x,y,x-y)dy\qquad \text{for }x\in\Omega
\end{align}
for suitable $u\colon \Omega\to \R$. Here the kernel $k\colon \Rn\times \Rn\times (\Rn\setminus\{0\})\to \R$ is  assumed to be $(d+2)$-times continuously differentiable and to satisfy the conditions 
\begin{alignat}{2}
      k(x,y,z)&=k(y,x,-z)\,, \label{k-ass-one} \\
     |\partial_x^\beta\partial_y^\gamma\partial_z^\delta k(x,y,z)| &\leqslant
        C_{\beta,\gamma,\delta}|z|^{-d-\alpha-|\delta|} \, , \label{k-ass-two} \\
     c_0 |z|^{-d-\alpha} &\leqslant k(x,y,z)\leqslant C_0 |z|^{-d-\alpha} \,. \label{k-ass-three}
   \end{alignat}for all $x,y,z\in\Rn$, $z\neq 0$ and $\beta, \gamma, \delta\in\N_0^d$ with $|\beta|+|\gamma|+|\delta|\leqslant d+2$ and some constants $C_{\beta,\gamma,\delta}, c_0,C_0>0$. Here $\alpha$ is the order of the operator, cf.~\cite{AK07}). We restrict ourselves to the case $\alpha \in (1,2)$. If $\omega\in C^{d+2}_{b}(\Rn)$, then $k(x,y,z) = \omega(x,y) |z|^{-d-\alpha}$ is an example of a kernel satisfying the previous assumptions.
   %In the case $\Omega = \Rn$ and $k(x,y,z) = |z|^{-d-\alpha}$ one has $\LL = const \times (-\Delta)^\frac{\alpha}{2}$ where $(-\Delta)^\frac{\alpha}{2}$ is a fractional power of the Laplace operator. If $\Omega$ is a bounded domain, $\LL$ has the same form as the generator of a censored stable process (cf., e.g., \cite{BBC03}) and it is also known as regional fractional Laplacian.

We add to our system  the boundary and initial conditions
\begin{alignat}{2}\label{eq:5}
 \bv|_{\partial \Omega} &= 0 &\qquad& \text{on}\ \partial\Omega\times (0,\infty), \\\label{eq:6}
 \partial_\no \mu|_{\partial \Omega} &= 0&& \text{on}\ \partial\Omega\times (0,\infty),  \\\label{eq:7}
 \left(\bv , \varphi \right)|_{t=0} &= \left( \bv_0 , \varphi_0 \right) &&\text{in}\ \Omega. 
\end{alignat}
Here $\partial_\no = \no\cdot \nabla$ and $\no$ denotes the exterior normal at $\partial\Omega$. We note that \eqref{eq:5} is the usual no-slip boundary condition for the velocity field and $\partial_\no \mu |_{\partial \Omega} = 0$ describes that there is no mass flux of the fluid components 
through the boundary. 
Furthermore we complete the system above by an additional boundary condition for $\varphi$, which will be part of the weak formulation, cf.\ Definition~\ref{defweaksolution} below.
If $\varphi$ is smooth enough (e.g.\, $\varphi(t)\in C^{1, \beta}(\overline{\Omega})$ for every $t\geq 0$) and $k$ fulfills suitable assumptions, then 
\begin{equation}\label{b.c.}
  \no_{x_{0}}\cdot   \nabla \varphi(x_{0})   = 0 \qquad \text{for all }x_0\in\partial\Omega
\end{equation}
where $\no_{x_{0}}$ depends on the interaction kernel $k$, cf. \cite[Theorem~6.1]{ABG15}, and $x_0\in\partial\Omega$.

The total energy of the system at time $t\geq 0$ is given by
\begin{align} \label{totalenergy}
 E_{\mbox{\footnotesize tot}}(\p,\bv)  
         = E_{\mbox{\footnotesize kin}}(\p,\bv) + E_{\mbox{\footnotesize free}}(\p)
\end{align}
where
\begin{equation*}
E_{\mbox{\footnotesize kin}}(\p,\bv)  = \int_\Omega \rho \frac{|\bv|^2}{2} \, dx,\qquad E_{\mbox{\footnotesize free}}(\p) = \int_\Omega \Psi(\p) \, dx  +  \frac12\E(\varphi,\varphi) 
\end{equation*}
are the kinetic energy and the free energy of the mixture, respectively, and
\begin{equation}\label{eq:DefnE}
    \E(u,v)= \int_\Omega\int_\Omega (u(x)-u(y))(v(x)-v(y))k(x,y,x-y)\sd x\sd y
\end{equation}
for all $u,v\in H^{\frac{\alpha}{2}}(\Omega)$ is the natural bilinear form associated to $\LL$, which will also be used to formulate the natural boundary condition for $\varphi$ weakly.
 Every sufficiently smooth solution of the system above satisfies the energy identity
\begin{equation*}
  \frac{d}{dt} E_{\mbox{\footnotesize tot}}(\p,\bv)= -\int_\Omega 2\eta(\varphi)|D\bv|^2\, dx - \int_\Omega m(\varphi)|\nabla \mu |^2\, dx
\end{equation*}
for all $t\geq 0$. This can be shown by testing \eqref{eq:1} with $\bv$, \eqref{eq:3} with $\mu$ and \eqref{eq:4} with $\partial_t \varphi$, where the product of $\LL \varphi$ and $\partial_t \varphi$ coincides with
\begin{equation*}
  \E(\varphi(t),\partial_t \varphi(t))
\end{equation*}
under \changes{the same natural boundary condition for $\varphi(t)$ as before, cf. \eqref{b.c.}}.

We consider
a class of singular free energies, 
which will be specified below and which includes the homogeneous free energy of the so-called regular solution models used by Cahn and Hilliard~\cite{CahnHilliard}:
\begin{equation}
\label{logpot}\changes{\Psi(\varphi) = \frac{\vartheta}2 \left((1+\varphi)\ln
(1+\varphi)+ (1-\varphi)\ln (1-\varphi)\right) - \frac{\vartheta_c}2 \varphi^2,\quad \varphi \in [-1,1]}
\end{equation}
where
$\changes{0<\vartheta<\vartheta_c}$. 
\changes{This choice of the free energies ensures that $\varphi(x,t)\in [-1,1]$ almost everywhere.} 
In order to deal with these terms we apply 
techniques, which were developed in Abels and Wilke~\cite{AW07} and extended to the present nonlocal Cahn-Hilliard equation in \cite{ABG15}.

Our proof of existence of a weak solution of \eqref{eq:1}-\eqref{eq:4}
together with a suitable initial and boundary condition follows closely the proof of the main result of \cite{ADG13}. The following are the main differences and difficulties of our paper compared with \cite{ADG13}.
Since we do not expect $H^1$-regularity in space for the volume fraction $\p$ of  a weak solution of our system, we should eliminate $\nabla \p$ from our weak formulation taking into account the incompressibility of $\bv$.
Implicit time discretization has to be constructed carefully, using a suitable
mollification of $\p$ and an addition of a small Laplacian term to the chemical potential equation taking into account of the lack of $H^1$-regularity in space of $\p$. While the arguments for the weak convergence of temporal interpolants of weak solutions of the time-discrete problem are similar to \cite{ADG13}, the function space used for the order parameter has less regularity in space 
since the nonlocal operator of order less than 2 is involved in the equation for the chemical potential. 
For the convergence of the singular term $\Psi'(\p)$, we employ the argument in \cite{ABG15}. The only difference is that we work in space-time domains directly. For the validity of the energy inequality, additional arguments using the equation of chemical potential and the fact that weak convergence together with norm convergence in uniformly convex Banach spaces imply strong convergence are needed.

The structure of the contribution is as follows: 
In Section~\ref{prelimi} we present some preliminaries, we fix notations and collect the
needed results on nonlocal operator. In Section~\ref{secexistence}, we define weak solutions
of our system and state our main result concerning the existence of weak solutions. In Section~\ref{sec:Implicit},  we define an implicit time discretization of our system and
show the existence of weak solutions of an associated time-discrete problem using the Leray-Schauder theorem. In Section~\ref{sec:proof}, we obtain compactness in time of temporal interpolants of the weak solutions of time-discrete problem and obtain weak solutions of our system as weak limits of a suitable subsequence.

%%%%%%%%%%%%%%%%%%%%%%%%%%%%%%%%%%%%%%%%%%%%%%%%%%%%%%%%%%%%%%%%%%%%%%%%%%%%%%%%%%%%%%%%%%%%%%%%
%%%%%%%%%%%%%%%%%%%%%%%%%%%%%%%%%%%%%%%%%%%%%%%%%%%%%%%%%%%%%%%%%%%%%%%%%%%%%%%%%%%%%%%%%%%%%%%%
%%%%%%%%%%%%%%%%%%%%%%%%%%%%%%%%%%%%%%%%%%%%%%%%%%%%%%%%%%%%%%%%%%%%%%%%%%%%%%%%%%%%%%%%%%%%%%%%
\section{Preliminaries} \label{prelimi}

As usual  %$\R^\dm_+=\{x\in \R^\dm: x_\dm >0\}$, 
$a\otimes b = (a_i b_j)_{i,j=1}^d$ for $a,b\in \R^d$  and $A_{\operatorname{sym}}= \frac12 (A+A^T)$ for  $A\in \R^{d\times d}$.
Moreover, 
\begin{equation*}
  \weight{f,g} \equiv \weight{f,g}_{X',X} = f(g), \qquad f\in X', g\in X
\end{equation*}
denotes the duality product, where $X$ is a Banach space and $X'$ is its duak.
We write $X\hookrightarrow \hookrightarrow Y$ if $X$ is compactly embedded into $Y$. 
For a Hilbert space $H$ its inner product is denoted by $(\cdot\,,\cdot )_H$.

\medskip

%\noindent {\bf Lebesgue and Sobolev spaces:}
Let $M\subseteq \R^d$ be measurable. As usual 
$L^q(M)$, $1\leq q \leq \infty$, denotes the Lebesgue space, $\|.\|_q$ its norm and $(.\,,.)_{M}=(.\,,.)_{L^2(M)}$ its inner product if $q=2$.   
Furthermore $L^q(M;X)$ denotes the set of all $f\colon M\to X$ that are strongly measurable and
$q$-integrable functions/essentially bounded functions. Here $X$ is a Banach
space. If $M=(a,b)$, we denote these spaces for simplicity by $L^q(a,b;X)$ and $L^q(a,b)$.
 Recall that $f\colon [0,\infty)\to X$ belongs $L^q_{\loc}([0,\infty);X)$ if and only if $f\in L^q(0,T;X)$ for every $T>0$. 
Furthermore, $L^q_{\uloc}([0,\infty); X)$ is the \emph{uniformly
  local} variant of $L^q(0,\infty;X)$ consisting of all strongly measurable $f\colon
[0,\infty)\to X$ such that
\begin{equation*}
  \|f\|_{L^q_{\uloc}([0,\infty); X)}= \sup_{t\geq 0}\|f\|_{L^q(t,t+1;X)} <\infty.
\end{equation*}
If $T<\infty$, we define $L^q_{\uloc}([0,T); X) := L^q(0,T;X)$.

For a domain $\Omega \subset \R^d$,  $m\in \N_0$, $1\leq q\leq \infty$, the standard Sobolev space is denoted by
$W^m_q(\Omega)$. % $W^m_{q,\loc}(\ol{\Omega})$ its local version, 
$W^m_{q,0}(\Omega)$  is the closure of $C^\infty_0(\Omega)$ in $W^m_q(\Omega)$,
$W^{-m}_q(\Omega)= (W^m_{q',0}(\Omega))'$, and $W^{-m}_{q,0}(\Omega)= (W^m_{q'}(\Omega))'$. 
% and $f\in W^{-m}_{q,\loc}(\ol{\Omega})$ if $f\in W^{-m}_q(\Omega\cap B)$ for
% every ball $B\subset \R^d$. 
 $H^s(\Omega)$ denotes the $L^2$-Bessel potential of order $s\geq 0$.

Let 
$
  f_\Omega = \frac1{|\Omega|}\int_\Omega f(x) \,dx
$
denote the mean value of $f\in L^1(\Omega)$. For $m\in\R$ we define
\begin{equation*}
  L^q_{(m)}(\Omega):=\{f\in L^q(\Omega):f_\Omega=m\}, \qquad 1\leq q\leq \infty.  
\end{equation*}
Then the orthogonal projection onto $L^2_{(0)}(\Omega)$ is given by 
\begin{align*}
 P_0 f:= f-f_\Omega= f-\frac1{|\Omega|}\int_\Omega f(x) \,dx\qquad \text{for all }f\in L^2(\Omega).
\end{align*}
For the following we denote
\begin{equation*}
 \Hone\equiv\Hone (\Omega)= H^1(\Omega)\cap L^2_{(0)}(\Omega), \qquad (c,d)_{\Hone(\Omega)} := (\nabla c,\nabla d)_{L^2(\Omega)}.  
\end{equation*}
Because of Poincar\'e's inequality, $\Hone(\Omega)$ is a Hilbert space. 
More generally, we define for $s \geq 0$
\begin{alignat*}{2}
\Hs[s]_{(0)} \equiv \Hs[s]_{(0)}(\Omega) &= \Hs[s](\Omega) \cap L^2_{(0)}(\Omega), &\quad \Hs[-s]_{(0)}(\Omega)&= (\Hs[s]_{(0)}(\Omega))', \\ 
\Hs[-s]_{0}(\Omega) &= (\Hs[\changes{s}](\Omega))',&\quad \Hs[-s](\Omega) &= (\Hs[s]_{0}(\Omega))'.
\end{alignat*}
Finally, $f\in \Hs[s]_{\loc}(\Omega)$ if and only if $f|_{\Omega'}\in H^s(\Omega')$ for every open and bounded subset $\Omega'$ with $\ol{\Omega'}\subset \Omega$.

%Definition of $\Hs[s]$ we add later.
\medskip

We denote by $L^2_\sigma(\Omega)$ is the closure of $C^\infty_{0,\sigma}(\Omega)$ in $L^2(\Omega)^d$, where 
$C^\infty_{0,\sigma}(\Omega)$ is the set of all
divergence free vector fields in $C^\infty_0(\Omega)^d$. The corresponding
Helmholtz projection, i.e., the $L^2$-orthogonal projection onto $L^2_\sigma(\Omega)$, is denoted by $P_\sigma$,
cf.\ e.g.\ Sohr \cite{Soh01}. % We note that $P_\sigma f = f- \nabla p$, 
% where $p \in W^1_2(\Omega)\cap L^2_{(0)}(\Omega)$ is the solution of the weak Neumann problem 
% \begin{equation}\label{eq:WeakHelmholtz}
%   (\nabla p,\nabla \varphi)_{\Omega} = (f, \nabla \varphi)\quad \text{for all}\ \varphi \in C^\infty(\ol{\Omega}).
% \end{equation}

\medskip

Let $I=[0,T]$ with $0<T< \infty$ or  $I=[0,\infty)$ if $T=\infty$ and let $X$ is a Banach
space. The Banach space of all bounded and continuous
$f\colon I\to X$ is denoted by $BC(I;X)$. It is equipped with the supremum norm. Moreover, $BUC(I;X)$ is defined as the
subspace of all bounded and uniformly continuous functions. Furthermore, $BC_w(I;X)$ is the set of all bounded and weakly
continuous $f\colon I\to X$. $C^\infty_0(0,T;X)$ denotes
the vector space of all smooth functions $f\colon (0,T)\to X$ with $\supp
f\subset\subset (0,T)$.
By definition $f\in W^1_p(0,T;X)$, $1\leq p <\infty$, if and only if $f,
\frac{df}{dt}\in L^p(0,T;X)$.% where $\frac{df}{dt}$ denotes the vector-valued
%distributional derivative of $f$.
Furthermore, $W^1_{p,\uloc}([0,\infty);X)$ is defined by replacing $L^p(0,T;X)$ by $L^p_{\uloc}([0,\infty);X)$ 
and we set $H^1(0,T;X)= W^1_2(0,T;X)$ and $H^1_{\uloc}([0,\infty);X) := W^1_{2,\uloc}([0,\infty);X)$.
Finally, we note:
\begin{lemma}\label{lem:CwEmbedding}
  Let $X,Y$ be two Banach spaces such that $Y\hookrightarrow X$ and $X'\hookrightarrow Y'$ densely.
  Then $L^\infty(I;Y)\cap BUC(I;X) \hookrightarrow BC_w(I;Y)$.
\end{lemma}
\noindent
For a proof see e.g. Abels \cite{Abe09a}.

\subsection{Properties of the Nonlocal Elliptic Operator $\mathcal{L}$}\label{S:nonlocal_operator}

In the following let $\E$ be defined as in \eqref{eq:DefnE}.
Assumptions~\eqref{k-ass-one}--\eqref{k-ass-three} yield that there are positive constants $c$ and $C$ such that
    \begin{equation*}
c\norm{H^{\frac{\alpha}{2}}(\Omega)}{u}^{2} \leqslant |\changes{u_{\Omega}}|^2 + \E(u,u) \leqslant  C \norm{H^{\frac{\alpha}{2}}(\Omega)}{u}^{2} \qquad \text{for all}\,u \in \Hs[\frac{\alpha}{2}](\Omega).
    \end{equation*}    
This implies that the following norm equivalences hold:
    \begin{alignat}{2}
        \mathcal{E}(u,u)&\sim \norm{\Hs[\frac{\alpha}{2}](\Omega)}{u}^2 &\qquad& \text{for all}\, u\in\Hs[\frac{\alpha}{2}]_{(0)}(\Omega),\\\label{eq:EquivNorm2}
        \mathcal{E}(u,u) + | \changes{u_{\Omega}} |^{2} &\sim \norm{\Hs[\frac{\alpha}{2}](\Omega)}{u}^2 &\qquad& \text{for all}\, u\in\Hs[\frac{\alpha}{2}](\Omega),
    \end{alignat}
cf.\ \cite[Lemma 2.4 and Corollary 2.5]{ABG15}.

In the following we will use a variational extension of the nonlocal linear operator $\mathcal{L}$ (see~\eqref{eq:defnL}) by defining $\mathcal{L} \colon \Hs[\frac{\alpha}{2}](\Omega) \to \Hs[-\frac{\alpha}{2}]_0(\Omega)$ as
\begin{equation*}
    \longduality{\Hs[-\frac{\alpha}{2}]_0}{\mathcal{L}u}{\varphi}{\Hs[\frac{\alpha}{2}]} = \mathcal{E}(u, \varphi) \quad \text{for all $\varphi \in \Hs[\frac{\alpha}{2}](\Omega)$}.
\end{equation*}
This implies
\begin{equation*}
    \duality{\mathcal{L}u}{1} = \mathcal{E}(u, 1) = 0.
\end{equation*}
 We note that $\mathcal{L}$ agrees with~\eqref{eq:defnL} as soon as $u \in \Hs[\alpha]_{\loc}(\Omega)\cap \Hs[\frac{\alpha}2](\Omega)$ and $\varphi\in C_0^\infty(\Omega)$, cf. \cite[Lemma 4.2]{AK07}. But this weak formulation also includes a natural boundary condition for $u$, cf. \cite[Theorem~6.1]{ABG15} for a discussion.

We will also need the following regularity result, which essentially states that the operator $\mathcal{L}$ is of lower order with respect to the usual Laplace operator. This result is from \cite[Lemma 2.6]{ABG15}.
\begin{lemma}\label{L:regularity_regularization}
    Let $g \in \Lp_{(0)}(\Omega)$ and $\theta>0$. Then the unique solution 
    $u \in \Hs_{(0)}(\Omega) $ for the problem
    \begin{equation}\label{eq:regularized_nonlocal}
        - \theta \int_{\Omega} \nabla u \cdot \nabla \varphi + \mathcal{E}(u, \varphi) = \Ltwoprod{g}{\varphi} \qquad \text{for all }\varphi \in \Hs_{(0)}(\Omega)
    \end{equation}
     belongs to $\Hs[2]_{\loc}(\Omega) $ and satisfies the estimate
    \begin{equation*}
        \theta \|\nabla u\|^2_{L^2(\Omega)} + \|u\|_{H^{\alpha/2}(\Omega)}^2 \leqslant C\|g\|_{L^2(\Omega)}^2,
    \end{equation*}
    where $C$ is independent of $\theta>0$ and $g$.
\end{lemma}

For the following let $\phi\colon [a,b]\to \R$  be continuous and define $\phi(x)=+\infty$ for $x\not\in [a,b]$. 
% The following regularity result is more involved and is from \cite[Lemma~2.7]{ABG15}. The proof in \cite{ABG15} is done by using ideas of the proof of \cite[Lemma~5.4]{AK07}. 
% \begin{lemma}\label{L:continuity}
%     Let  $\Omega\subseteq \R^d$ be a bounded domain with $C^2$-boundary, $\phi$ as before, and let $u\in H^{\frac{\alpha}2}(\Omega)$ such that $\phi'(u)\in L^2(\Omega)$ and
%     \begin{equation}\label{eq:Weak}
%         \mathcal{E}(u,\varphi) +\int_\Omega \phi'(u)\varphi \, dx  =\int_\Omega g\varphi \, dx \qquad \text{for all}\ \varphi\in H^{\frac{\alpha}2}(\Omega)
%     \end{equation}
%     for some given $g\in H^1(\Omega)$. Then $u\in \Cont[\beta](\overline\Omega)$ for some $\beta\in (0,1)$ depending only on $d$ and there is a constant $C>0$ independent of $u$ and $g$ such that
%     \begin{equation}\label{eq:HoelderEstim}
%         \|u\|_{\Cont[\beta](\overline\Omega)} \leqslant C\left(\|g\|_{H^1(\Omega)} + \|u\|_{H^{\alpha/2}(\Omega)}+\|\phi'(u)\|_{L^2(\Omega)}\right).
%     \end{equation}
% \end{lemma}
 As in \cite[Section~3]{ABG15} we fix $\theta\geqslant 0$
and consider the functional
\begin{equation}\label{eq:DefnF}
    F_{\theta} (c) = \frac{\theta}{2} \int_{\Omega} |\nabla c|^2 \sd x + \frac12\E(c,c)  + \int_{\Omega} \phi( c(x)) \sd x
\end{equation}
where
\begin{eqnarray*}
    \dom F_0&=& \left\{c\in H^{\alpha/2}(\Omega)\cap L^2_{(m)}(\Omega): \phi(c)\in L^1(\Omega)\right\},\\
    \dom F_\theta&=& H^1(\Omega)\cap \dom F_0\qquad \text{if}\ \theta >0
\end{eqnarray*}
for a given $m\in(a,b)$.
Moreover, we define
\begin{equation*}
    \E_\theta (u,v)= \theta \int_\Omega \nabla u\cdot \nabla v \sd x +\E(u,v)
\end{equation*}
for all $u,v\in H^1(\Omega)$ if $\theta>0$ and $u,v\in H^{\alpha/2}(\Omega)$ if $\theta = 0$.

In the following $\partial F_\theta(c)\colon L^2_{(m)}(\Omega)\to \mathcal{P}(L^2_{(0)}(\Omega))$ denotes the subgradient of $F_\theta$ at $c\in \dom F$, i.e.,  $w\in \partial F_\theta(c)$ if and only if
\begin{equation*}
    (w,c'-c)_{L^2} \leqslant F_\theta(c')-F_\theta(c)\qquad \text{for all }c'\in L^2_{(m)}(\Omega).
\end{equation*}
%We note that $L^2_{(m)}(\Omega)$ is an affine subspace of $L^2(\Omega)$ with tangent space $L^2_{(0)}(\Omega)$ and this is a simple generalization of the standard definition of $\partial F$ for functionals on Hilbert spaces to affine subspaces of Hilbert spaces. By a simple translation one can always reduce to a case of linear subspace.

% First of all let us prove the following
% \begin{lemma}\label{lem:LSC}
%     Let $\phi\colon [a,b]\to \R$ be a continuous and convex function. Then, for any $\theta \geqslant 0$, $F_{\theta}$ defined as in~\eqref{eq:DefnF} is a proper, lower semicontinuous, convex functional.
% \end{lemma}
% \begin{cor}\label{cor:MaximalMonotone}
%     Let $\phi$ and $F_{\theta}$ be as in Lemma~\ref{lem:LSC} and let $m=0$. Then, for any $\theta \geqslant 0$, $\partial F_{\theta}$ is a maximal monotone operator on $H = \Lp_{(0)}(\Omega)$.
% \end{cor}

The following characterization of $\partial F_\theta(c)$ is an important tool for the existence proof.
\begin{theorem}\label{thm:Regularity}
    Let $\phi\colon [a,b]\to \R$ be a convex function that is twice continuously differentiable in $(a,b)$ and satisfies $\lim_{x\to a} \phi'(x)=-\infty$, $\lim_{x\to b} \phi'(x) =+\infty $. Moreover, we set $\phi(x)=+\infty$ for $x\not\in (a,b)$ and let $F_\theta$ be defined as in~\eqref{eq:DefnF}. Then $\partial F_\theta\colon \mathcal{D}(\partial F_\theta)\subseteq L^2_{(m)}(\Omega)\to L^2_{(0)}(\Omega)$ is a single valued, maximal monotone operator with
    \begin{eqnarray*}
        \mathcal{D}(\partial F_0) &=&\Big\{ c\in H^\alpha_{\loc}(\Omega)\cap H^{\alpha/2}(\Omega)\cap L^2_{(m)}(\Omega): \phi'(c)\in L^2(\Omega),\exists f\in L^2(\Omega):\\
        & & \quad  \E(c,\varphi) + \int_\Omega \phi'(c)\varphi \, dx= \int_{\Omega} f\varphi \sd x \quad \forall \, \varphi \in H^{\alpha/2}(\Omega)
        \Big\}
    \end{eqnarray*}
    if $\theta=0$ and
    \begin{eqnarray*}
        \mathcal{D}(\partial F_\theta) &=&\Big\{ c\in H^2_{\loc}(\Omega) \cap H^1(\Omega) \cap L^2_{(m)}(\Omega): \phi'(c)\in L^2(\Omega),\exists f\in L^2(\Omega):\\
        & &\quad  \E_\theta(c,\varphi) + \int_\Omega \phi'(c)\varphi \, dx= \int_{\Omega} f\varphi \sd x \quad\, \forall\,\varphi \in H^1(\Omega) \Big\}
    \end{eqnarray*}
    if $\theta > 0$ as well as
    \begin{equation*}
        \partial F_\theta (c) = -\theta  \Delta c + \LL c + P_0\phi'(c)\quad \text{in}\ \mathcal{D}'(\Omega) \qquad \text{for $\theta \geqslant 0$.}
    \end{equation*}
    Moreover, the following estimates hold
    \begin{alignat}{1}\label{eq:DomEstim}
          \theta \|c\|_{H^1}^2+ \|c\|_{H^{\alpha/2}}^2+ \|\phi'(c)\|_2^2&\leqslant  C\left(\|\partial F_\theta(c)\|_2^2 + \|c\|_2^2+1\right)\\\nonumber
        \int_\Omega\int_\Omega (\phi'(c(x))-\phi'(c(y)))&(c(x)-c(y))k(x,y,x-y)\sd x\sd y\\\nonumber
        &\leqslant  C\left(\|\partial F_\theta(c)\|_2^2 + \|c\|_2^2+1\right)\\\nonumber
        \theta \int_\Omega \phi''(c)|\nabla c|^2 \sd x &\leqslant C\left(\|\partial F_\theta(c)\|_2^2 + \|c\|_2^2+1\right)
    \end{alignat}
    for some constant $C>0$ independent of $c\in \mathcal{D}(\partial F_\theta) $ and $\theta \geqslant 0$.
\end{theorem}
The result follows from \cite[Corollary~3.2 and Theorem~3.3]{ABG15}.

%%%%%%%%%%%%%%%%%%%%%%%%%%%%%%%%%%%%%%%%%%%%%%%%%%%%%%%%%%%%%%%%%%%%%%%%%%%%%%%%%%%%%%%%%%%%%%%%%%%%%%
%%%%%%%%%%%%%%%%%%%%%%%%%%%%%%%%%%%%%%%%%%%%%%%%%%%%%%%%%%%%%%%%%%%%%%%%%%%%%%%%%%%%%%%%%%%%%%%%%%%%%%
%%%%%%%%%%%%%%%%%%%%%%%%%%%%%%%%%%%%%%%%%%%%%%%%%%%%%%%%%%%%%%%%%%%%%%%%%%%%%%%%%%%%%%%%%%%%%%%%%%%%%%
\section{Weak Solutions and Main Result} \label{secexistence}

In this section we define weak solutions for the system \eqref{eq:1}-\eqref{eq:4}, \eqref{eq:5}-\eqref{eq:7} together with a natural boundary condition for $\varphi$ given by the bilinear form $\E$, summarize the assumptions and state the main result. 

% of the following
% Navier-Stokes/ Cahn-Hilliard system for a situation with different densities. The complete
% system is given by 
% \begin{align}
%  \partial_t (\rho \mathbf{v}) + \operatorname{div} (\rho \bv \otimes \bv) &- \operatorname{div} (2 \eta(\varphi) D \bv)
%   + \nabla p + \mbox{div}(\bv \otimes \tbj) \nonumber \\
%   & = \mu \nabla \varphi & \mbox{in } \, Q , \label{newproblem1} \\
%  \operatorname{div} \, \bv &= 0 & \mbox{in } \, Q , \label{newproblem2} \\
%  \partial_t \varphi + \bv \cdot \nabla \varphi &= \mbox{div}\left(m(\varphi) \nabla \mu \right) & \mbox{in } \, Q , \label{newproblem3} \\
%  \mu &= \Psi'(\varphi) +   \LL \varphi  & \mbox{in } \, Q , \label{newproblem4} \\
%  \bv|_{\partial \Omega} &= 0 & \mbox{on } \, S , \label{newproblem5} \\
%   \partial_n \mu|_{\partial \Omega} &= 0 & \mbox{on } \, S , \label{newproblem6} \\
%  \left(\bv , \varphi \right)|_{t=0} &= \left( \bv_0 , \varphi_0 \right) & \mbox{in } \, \Omega , \label{newproblem7}
% \end{align}
% where $\tbj = -\frac{\tilde{\rho}_2 - \tilde{\rho}_1}{2} m(\varphi) \nabla \mu$. %To simplify the 
% %presentation we set $\hat{\sigma} = \varepsilon = 1$, but the result will also be true for 
% %general $\hat{\sigma},$ $\varepsilon > 0$.
%In the above formulation and in the following, we use the abbreviations for space-time cylinders
%$Q_{(s,t)}=\Omega \times (s,t)$, $Q_t = Q_{(0,t)}$ and $Q = Q_{(0,\infty)}$ and analogously for the 
%boundary $S_{(s,t)}=\partial \Omega \times (s,t)$, $S_t = S_{(0,t)}$ and $S = S_{(0,\infty)}$.

\begin{assumption} \label{assumptions}
Let $\Omega \subset \R^d$, $d=2,3$, be a bounded domain with $C^2$-boundary. The following conditions hold true:
\begin{enumerate}
 \item $\rho(\p) = \frac{1}{2}(\tilde{\rho}_1 + \tilde{\rho}_2) + \frac{1}{2} (\tilde{\rho}_2 - \tilde{\rho}_1) \p$ for all $\varphi \in [-1,1]$.% where $\tilde{\rho}_i>0$ are the specific constant mass 
       %densities of the unmixed fluids and $\varphi$ is the difference of the volume fractions of the fluids, cf.~\cite{AGG11}. 
 \item $m \in C^1(\R)$, $\eta \in C^0(\R)$ and there are constants $m_0,K>0$ such that $0 < m_0 \leq m(s),\eta(s) \leq K$ for 
       all $s\in\R$.
 \item  
       $\Psi \in C([-1,1]) \cap C^2((-1,1))$ and 
       \begin{align} \label{assumptionsphi}
         \lim_{s \to \pm 1} \Psi'(s) = \pm\infty \,, \quad
         \Psi''(s) \geq -\kappa \; \mbox{ for some $\kappa \in \R$} \,.
       \end{align}
% \item Additionally we impose the condition that $\lim_{s \rightarrow \pm 1} 
 %      \frac{\Psi''(s)}{|\Psi'(s)|} = +\infty$. % Note: Reference on this item without labels in the next remark!
\end{enumerate}
\end{assumption}
A standard example for a homogeneous free energy density $\Psi$ satisfying the previous conditions is  given by \eqref{logpot}.
Since for solutions we will have $\varphi(x,t)\in [-1, 1]$ almost everywhere, we only need the functions $m,\eta$
        on this interval. But for simplicity we assume $m,\eta$ to be defined on  $\mathbb{R}$.

\begin{definition} \label{defweaksolution}
 Let  $\bv_0 \in L^2_\sigma(\Omega)$ and $\p_0 \in H^{\alpha/2}(\Omega)$ with $|\varphi_0| \leq 1$ almost everywhere in $\Omega$ and let Assumption \ref{assumptions} be satisfied. Then $(\bv,\p,\mu)$ 
such that
 \begin{align*}
  & \bv \in BC_w([0,\infty);L^2_\sigma(\Omega)) \cap L^2(0,\infty;H_0^1(\Omega)^d) \,, \\
  & \p \in BC_w([0,\infty);H^{\alpha/2}(\Omega)) \cap L^2_{\mbox{\footnotesize uloc}}([0,\infty);H^\alpha_{\loc}(\Omega)) \,, \; \
        \Psi'(\p) \in L^2_{\mbox{\footnotesize uloc}}([0,\infty);L^2(\Omega)) \,, \\
  & \mu \in L^2_{\mbox{\footnotesize uloc}}([0,\infty);H^1(\Omega)) 
    \; \mbox{ with } \; \nabla \mu \in L^2(0,\infty;L^2(\Omega)) 
 \end{align*}
 is called a weak solution of \eqref{eq:1}-\eqref{eq:4}, \eqref{eq:4}-\eqref{eq:5}
 if the following conditions hold true:
 \begin{align}  
   - \left(\rho \bv , \partial_t \bpsi \right)_{Q} 
  &+ \left( \operatorname{div}(\rho \bv \otimes \bv) , \bpsi \right)_{Q}
  + \left(2 \eta(\p) D\bv , D\bpsi \right)_{Q} 
   - \left( (\bv \otimes \tbj) , \nabla \bpsi \right)_{Q} \nonumber \\
  &= -\left( \varphi\nabla \mu  , \bpsi \right)_{Q} \label{weakline1} 
 \end{align}
 for all $\bpsi \in C_0^\infty(\Omega \times (0,\infty))^d$ with $\operatorname{div} \bpsi = 0$,
 \begin{align} 
  - \left(\p , \partial_t \psi \right)_{Q} 
  + \left( \bv \cdot \nabla \p , \psi \right)_{Q}
  &= - \left(m(\p) \nabla \mu , \nabla \psi \right)_{Q}  \label{weakline2} \\
 \int_0^\infty \int_\Omega \mu \psi \, dx \, dt =  \int_0^\infty \int_\Omega\Psi'(\varphi) \psi \, dx\, dt  &+ \int_0^\infty \E(\varphi(t),\psi(t))\, dt \label{weakline3}
 \end{align}
 for all $\psi \in C_0^\infty((0,\infty);C^1(\overline{\Omega}))$ and
 \begin{align}
  \left.\left( \bv,\p \right)\right|_{t=0} &= \left( \bv_0 , \p_0 \right) \,. \label{weakline4} 
 \end{align}
\changes{Recall $ \tbj = -\frac{\tilde{\rho}_2 - \tilde{\rho}_1}{2} m(\varphi) \nabla \mu.$}
 Finally, the energy inequality
 \begin{align} 
  E_{\mbox{\footnotesize tot}}(\p(t),\bv(t)) &+ \int_s^t\int_{\Omega} 2 \eta(\p) \, |D\bv|^2 \, dx\, d\tau 
        + \int_s^t\int_{\Omega} m(\p) |\nabla \mu|^2 \, dx\, d\tau \nonumber \\
   &\leq E_{\mbox{\footnotesize tot}}(\p(s),\bv(s)) \label{weakline5}
 \end{align}
 holds true for all $t \in [s,\infty)$ and almost all $s \in [0,\infty)$ (including $s=0$). Here $E_{\mbox{\footnotesize tot}}$ is as in \eqref{totalenergy}.
\end{definition}

The main result of this contribution is:
\begin{theorem}[Existence of Weak Solutions] \label{existenceweaksolution}~\\
 Let Assumption \ref{assumptions} hold \changes{and $\alpha\in (1,2)$}. 
 Then for every  $\bv_0 \in L^2_\sigma(\Omega)$ and $\varphi_0 \in H^{\alpha/2}(\Omega)$
 such that $|\varphi_0| \leq 1$ almost everywhere and $\changes{(\varphi_0)_{\Omega}} \in (-1,1)$ there exists a weak solution $(\bv,\varphi,\mu)$ of \eqref{eq:1}-\eqref{eq:4}, \eqref{eq:5}-\eqref{eq:7}. 
% in the sense of Definition \ref{defweaksolution}.
\end{theorem}\changes{
\begin{remark}
  Using e.g.\ $\varphi \nabla \mu \in L^2(0,\infty; L^2(\Omega))$ one can consider this term in \eqref{weakline1} as a given right-hand side and obtain the existence of a pressure such that \eqref{eq:1} holds in the sense of distributions in the same way as for the single Navier-Stokes equations, cf. e.g.\ \cite{Soh01}.
\end{remark}}

\section{Approximation by an Implicit Time Discretization}\label{sec:Implicit}

Let $\Psi$ be as in Assumption~\ref{assumptions}. We define $\Psi_0\colon [-1,1]\to \R$ by $\Psi_0(s)= \Psi(s)+ \kappa \frac{s^2}2$ for all $s\in [a,b]$. Then $\Psi_0\colon [-1,1]\to\R$ is convex and $\lim_{s\to\pm 1}\Psi'_0(s)= \pm \infty$. A basic idea for the following is to use this decomposition to split the free energy $E_{\mbox{\footnotesize free}}$ into a singular convex part $E$ and a quadratic perturbation. In the equations this yields a decomposition into a singular monotone operator and a linear remainder.
To this end we define an energy $E \colon L^2(\Omega) \to \R\cup\{+\infty\}$ 
with domain $$\mbox{dom}\, E = \{\p \in H^{\alpha/2}(\Omega) \;|\; -1 \leq \p \leq 1 \,\mbox{ a.e.} \}$$ 
given by 
\begin{align} \label{helpenergy}
 E(\p) = \left\{ 
                    \begin{array}{cl} 
                     \frac12\E(\p,\p) + \int_\Omega \Psi_0(\p) \, dx
                        & \mbox{for } \; \varphi \in \mbox{dom}\, E \,, \\
                     + \infty & \mbox{else} \,.
                    \end{array}
                    \right.
\end{align}
This yields the decomposition
\begin{align*}
 E_{\mbox{\footnotesize free}}(\p) &= E(\p) - \frac{{\kappa}}{2} \|\p\|_{L^2}^2 \qquad \text{for all }\p\in \mbox{dom}\ E.
\end{align*}
Moreover, $E$ is convex and $E=F_0$ if one chooses $\phi=\Psi_0$ and $F_0$ is as in Subsection~\ref{S:nonlocal_operator}. This is a key relation for the following analysis in order to make use of Theorem~\ref{thm:Regularity}, which in particular implies that $\partial E=\partial F_0$ is a maximal monotone operator.

%\subsection{Definition of the Time-Discrete Problem} \label{subsec:DefImplicit}

To prove our main result we discretize our system semi-implicitly in time in a suitable manner.
To this end, let $h=\frac{1}{N}$ for $N \in \mathbb{N}$ and $\bv_k \in L^2_\sigma(\Omega)$, 
$\p_k \in H^1(\Omega)$ with $\p_k(x)\in [-1,1]$ almost everywhere  and $\rho_k = 
\frac{1}{2}(\tilde{\rho}_1 + \tilde{\rho}_2) + \frac{1}{2} (\tilde{\rho}_2 - \tilde{\rho}_1) \p_k$ 
be given. Then $\Psi(\p_k)\in L^1(\Omega)$. We also define a smoothing operator $P_h$ on $L^2(\Omega)$ as follows.
We choose $u$ as the solution of the following heat equation
\begin{eqnarray*}
 \left\{ \begin{array}{rcll}
          \partial_t u - \Delta u &=& 0 & \mbox{in } \, \Omega \times (0,T) \,, \\
          u|_{t=0} &=& \p' & \mbox{on } \, \Omega \,, \\
          \left.\partial_\nu u\right|_{\partial \Omega} &=& 0 & \mbox{on } \, \partial \Omega \times (0,T), 
         \end{array}
 \right.
\end{eqnarray*}
where $ \p' \in L^2(\Omega) $, and set $ P_h \p' := u|_{t=h} $. Then $P_h \p'
\in H^2(\Omega)$ and $P_h \p' \rightarrow \p'$ in $L^2(\Omega)$ as $h\to 0$ for all $\p'\in L^2(\Omega)$. Moreover, we have $|P_h \p'| \leq 1$ in $\Omega$ if $|\p'(x)|\leq 1$ almost everywhere and $P_h \p' \rightarrow_{h\to 0} \p'$ in $H^{\frac{\alpha}{2}}(\Omega)$ as $h\to 0$ for all $\p'\in H^{\frac{\alpha}{2}}(\Omega)$. 

Now we determine $(\bv,\p,\mu)=(\bv_{k+1},\p_{k+1},\mu_{k+1})$, $k\in\N$, successively as solution of the 
following  problem:
Find $\bv \in H_0^1(\Omega)^d \cap L^2_\sigma(\Omega)$, $\varphi \in \mathcal{D}(\partial E)$
and 
$$
\mu \in H^2_n(\Omega) = \{u \in H^2(\Omega) \,|\, \left. \partial_\no u \right|_{\partial \Omega} = 0 \mbox{ on } \partial \Omega\},
$$ 
such that
\begin{align} 
 \left( \frac{\rho \bv - \rho_k \bv_k}{h} , \bpsi\right)_\Omega 
 &+ \left( \mbox{div}(\rho(P_h \p_k)  \bv \otimes \bv) , \bpsi \right)_{\Omega}
 + \left(2 \eta(\p_k) D\bv , D \bpsi \right)_\Omega 
 + \left( \mbox{div}( \bv \otimes \tbj) , \bpsi \right)_{\Omega} \nonumber \\
 &= -\left((P_h \p_k) \nabla \mu , \bpsi \right)_\Omega \label{timediscretizationline1}
\end{align}
for all $\bpsi \in C_{0,\sigma}^\infty(\Omega)$, 
%(or equivalently $\bpsi \in H^1_0(\Omega)^d \cap L^2_\sigma(\Omega)$), 
\begin{align} 
 &\frac{\p - \p_k}{h} + \bv \cdot \nabla P_h \p_k = \mbox{div} \left( m(P_h \p_k) \nabla \mu \right) 
  \; \mbox{ almost everywhere in $\Omega$ },  \label{timediscretizationline2}
\end{align}
and 
\begin{align}
 &  \int_\Omega \left(\mu + {\kappa} \, \frac{\p + \p_k}{2}\right)\psi \, dx 
 = \E(\p,\psi) +\int_\Omega {\Psi}_0'(\p)\psi\, dx + h \int_\Omega \nabla \p \cdot \nabla \psi \, dx \label{timediscretizationline3}
\end{align}
for all $\psi\in H^{\alpha/2}(\Omega)$, where
\begin{align*}
 \tbj \equiv \tbj_{k+1} := - \tfrac{\tilde{\rho}_2 - \tilde{\rho}_1}{2} m(P_h \varphi_k) \nabla \mu_{k+1}
  = - \tfrac{\tilde{\rho}_2 - \tilde{\rho}_1}{2} m(P_h \varphi_k) \nabla \mu \,.
\end{align*}
For the following let 
\begin{align}
	 E_{\mbox{\footnotesize tot}, h}(\p, \bv)=\int_{\Omega} \rho \frac{|\bv|^2}{2} \, dx + \int_{\Omega} \Psi(\p) \, dx + \frac12 \E(\p, \p) + \frac{h}{2} \int_{\Omega} |\nabla \p|^2 \, dx.
\end{align} 
denote the total energy of the system \eqref{timediscretizationline1}-\eqref{timediscretizationline3}.

\begin{remark} %\vskip -10pt
 \begin{enumerate}
  \item As in \cite{ADG13} we obtain the important relation
        \begin{align*}
         -\frac{\rho - \rho_k}{h} - \bv \cdot \nabla \rho(P_h \p_k) = \operatorname{div} \tbj \,,
        \end{align*}
        by multiplication of \eqref{timediscretizationline2} with $-\frac{\tilde{\rho}_2 - \tilde{\rho}_1}{2}= \frac{\partial\rho(\p)}{\partial \p}$.
        Because of $\operatorname{div}(\bv \otimes \tbj) = (\operatorname{div} \tbj) \bv 
        + \left(\tbj \cdot \nabla \right) \bv$ this yields that
        \begin{align} \label{equivtimediscretizationline1}
         &\left( \frac{\rho \bv - \rho_k \bv_k}{h} , \bpsi\right)_\Omega 
          + \left( \operatorname{div}(\rho(P_h \p_k)  \bv \otimes \bv) , \bpsi \right)_{\Omega}
          + \left(2 \eta(\p_k) D\bv , D \bpsi \right)_\Omega \\
          +& \left( \left(\operatorname{div} \tbj - \frac{\rho - \rho_k}{h} - \bv \cdot \nabla \rho(P_h \p_k) \right) \frac{\bv}{2} 
              , \bpsi \right)_{\Omega} 
          + \left( \left( \tbj \cdot \nabla \right) \bv , \bpsi \right)_{\Omega} 
          = - \left( (P_h \p_k) \nabla \mu , \bpsi \right)_\Omega \nonumber  
        \end{align}
        for all $\bpsi \in C_{0,\sigma}^\infty(\Omega)$ to \eqref{timediscretizationline1}, which will be used to derive suitable 
        a-priori estimates.
  \item Integrating \eqref{timediscretizationline2} in space one obtains $\int_\Omega \varphi \, dx = \int_\Omega \varphi_k \, dx$  because of $\operatorname{div} \, \bv = 0$ and the boundary
        conditions.
 \end{enumerate}
\end{remark}

The following lemma is important to control the derivative of the singular free energy density $\Psi'(\p)$.
\begin{lemma} \label{derivativeforchempot}
 Let $\varphi \in \mathcal{D}(\partial F_h)$ and 
 $\mu \in H^1(\Omega)$ be a solution of \eqref{timediscretizationline3} for given $\varphi_k \in H^1(\Omega)$ 
 with $|\varphi_k(x)| \leq 1$ almost everywhere in $\Omega$ such that
 \begin{align*}
  \changes{\varphi_{\Omega}=}\tfrac{1}{|\Omega|} \int_\Omega \varphi \, dx = \tfrac{1}{|\Omega|} \int_\Omega \varphi_k \, dx \in (-1,1) \,.
 \end{align*}
 Then there is a constant $C=C(\int_\Omega \varphi_k, \Omega)>0$, independent of $\p, \mu, \p_k$, such that
 \begin{align*}
  \| \Psi_0'(\p) \|_{L^2(\Omega)} + \left| \int_\Omega \mu \, dx \right| 
   & \leq C (\|\nabla \mu\|_{L^2} + \|\nabla \p\|_{L^2}^2 + 1) \; \mbox{and} \\
  \|\partial F_h(\p) \|_{L^2(\Omega)} & \leq C \left( \|\mu\|_{L^2} + 1 \right) \,. 
 \end{align*}
 %if $\|(\p,\p_k)\|_{H^1} \leq R$.
\end{lemma}
\begin{proof}
  The proof is an adaptation of the corresponding result in \cite{ADG13}. For the convenience of the reader we give the details.
First we choose  $\psi = \p - \changes{\p_{\Omega}}$ in \eqref{timediscretizationline3} and  get
\begin{align} \label{testedwithFanddiff}
  &\int_\Omega \mu (\p-\changes{\p_{\Omega}} ) \, dx
 + \int_\Omega \kappa \frac{\p + \p_k}{2} (\p - \changes{\p_{\Omega}}) \, dx \nonumber \\
 =& \E(\p, \p)  
   + \int_\Omega \Psi_0'(\p) (\p - \changes{\p_{\Omega}}) \, dx \, +h \int_{\Omega} \nabla \p \cdot \nabla \p \, dx \,.
\end{align}
Let $\mu_0 = \mu - \changes{\mu_{\Omega}}$. Then $\int_\Omega \mu (\p - \changes{\p_{\Omega}}) \, dx = \int_\Omega \mu_0 \p \, dx$.

In order to estimate the second term in \eqref{testedwithFanddiff} we 
use that $\overline{\p} \in (-1+\varepsilon,1-\varepsilon)$ for sufficiently small $\varepsilon > 0$ and that $\lim_{\p \to \pm1} \Psi_0'(\p) = \pm \infty$. Hence for sufficiently small $ \varepsilon$ one obtains the inequality
$\Psi_0'(\p) (\p - \changes{\p_{\Omega}}) \geq C_\varepsilon |\Psi_0'(\p)| - \tilde{C}_\varepsilon$, which implies 
\begin{align*}
 \int_\Omega \Psi_0'(\p) (\p - \changes{\p_{\Omega}}) \, dx
   \geq C \int_\Omega |\Psi_0'(\p)| \, dx - C_1 \,.
\end{align*}
Together with \eqref{testedwithFanddiff} we obtain
\begin{align*}
 \int_\Omega |\Psi_0'(\p)| \, dx
  &\leq \; C \|\mu_0\|_{L^2(\Omega)} \|\p\|_{L^2(\Omega)} 
  + C \int_{\Omega} \frac{\kappa}{2} | \p + \p_k| | \p - \changes{\p_{\Omega}} | \, dx
  %+ C(h\|\nabla \p\|_{L^2(\Omega)}^2 + \E(\p, \p)) 
  + C_1 \\
  &\leq  \; C (\|\mu_0\|_{L^2(\Omega)}  + \|\p\|_{L^2(\Omega)}^2 + 1) \\
  &\leq  \; C (\|\nabla \mu\|_{L^2(\Omega)} + 1)  \,,
  %\leq & \; C(\|\p\|_{H^1},\|\p_k\|_{H^1}) \, (\|\mu_0\|_{L^2} + 1) 
  %\; \leq \; C(\|\p\|_{H^1},\|\p_k\|_{H^1}) \, (\|\nabla \mu\|_{L^2} + 1) \,.
\end{align*}
because of $|\varphi|$, $|\varphi_k| \leq 1$. 
Next we choose  $ \psi \equiv 1$ in \eqref{timediscretizationline3}. This yields
\begin{align*}
 \int_\Omega \mu \, dx = 
  \int_\Omega \Psi_0'(\p) \, dx 
  - \int_\Omega \frac{\kappa}{2} \left(\p + \p_k \right) \, dx \,.
\end{align*}
Altogether this leads to 
\begin{align*}
 \left| \int_\Omega \mu \, dx \right| \leq & 
   C (\|\nabla \mu\|_{L^2(\Omega)} + 1) \,.
\end{align*}
Finally, the estimates of  $\partial F_h(\p)$ and  $\Psi_0'(\p)$ in  $L^2(\Omega)$
follow directly from \eqref{timediscretizationline3} and \eqref{eq:DomEstim}.
\end{proof}

Now we will prove existence of solution to the time-discrete system.
We basically follow the line of the corresponding arguments in \cite{ADG13} here. As before we denote
\begin{equation*}
  H^2_n(\Omega):=\{u\in H^2(\Omega):\no\cdot \nabla u|_{\partial\Omega}=0\}.
\end{equation*}
\begin{lemma} \label{timediscreteexistence}
 For every $\bv_k \in L^2_\sigma(\Omega)$, $\varphi_k \in H^1(\Omega)$ with $|\varphi_k(x)|\leq 1$ almost everywhere, %with $\Psi'(\varphi_k) \in L^2(\Omega)$ 
 and $\rho_k = \frac{1}{2} (\tilde{\rho}_1 + \tilde{\rho}_2) + \frac{1}{2} (\tilde{\rho}_2 - \tilde{\rho}_1) \varphi_k$
 there is some solution $(\bv,\p,\mu) \in \left( H^1_0(\Omega)^d \cap L^2_\sigma(\Omega) \right)
 \times \mathcal{D}(\partial F_h) \times H^2_n(\Omega)$ of the system \eqref{timediscretizationline2}-\eqref{timediscretizationline3} and \eqref{equivtimediscretizationline1}. Moreover, the solution
 satisfies the discrete energy estimate
 \begin{align} \label{discreteenergyestimate}
  E_{\mbox{\footnotesize tot,h}}&(\p,\bv) + \int_\Omega \rho_k \frac{|\bv - \bv_k|^2}{2} \, dx
    + \int_\Omega \frac{|\nabla \p - \nabla \p_k|^2}{2} \, dx 
    + \frac12 \E(\p - \p_k, \p - \p_k)
    \nonumber \\
    &+ h \int_\Omega 2 \eta(\p_k) |D\bv|^2 \, dx + h \int_\Omega m(\p_k) |\nabla \mu|^2 \, dx 
    \leq E_{\mbox{\footnotesize tot,h}}(\p_k,\bv_k) \,. 
 \end{align}
\end{lemma}
\begin{proof}
As first step we prove the energy estimate \eqref{discreteenergyestimate} for any solution 
$(\bv,\p,\mu) \in \left( H^1_0(\Omega)^d \cap L^2_\sigma(\Omega) \right) \times \mathcal{D}(\partial F_h) \times H^2_n(\Omega)$ 
of \eqref{timediscretizationline2}-\eqref{timediscretizationline3} and 
\eqref{equivtimediscretizationline1}.

We choose   $\bpsi = \bv$ in \eqref{equivtimediscretizationline1} and use that
\begin{align*}
 \int_\Omega \left( (\mbox{div} \, \tbj) \frac{\bv}{2} + \left( \tbj \cdot \nabla \right) \bv \right) \cdot \bv \, dx
  = \int_\Omega \mbox{div} \left( \tbj \frac{|\bv|^2}{2} \right) \, dx = 0. 
\end{align*}
Then we derive as in \cite[Proof of Lemma 4.3]{ADG13}
\begin{align*}
 \int_\Omega & \left( \mbox{div}(\rho(P_h \p_k) \bv \otimes \bv) - (\nabla \rho(P_h \p_k) \cdot \bv) \frac{\bv}{2} \right) \cdot \bv \, dx 
% &= \int_\Omega \left( \mbox{div} (\rho(P_h(\p_k) \bv) |\bv|^2 
%      + \rho(P_h \p_k) \bv \cdot \nabla \left( \frac{|\bv|^2}{2} \right) 
%      - \mbox{div}(\rho(P_h \p_k) \bv) \frac{|\bv|^2}{2} \right) dx\\
 = \int_\Omega \mbox{div} \left( \rho(P_h \p_k) \bv \frac{|\bv|^2}{2} \right) \, dx 
 = 0 \,,
\end{align*}
due to $\operatorname{div} \bv = 0$. Next %from the simple algebraic equation
%\begin{align*}
% \mathbf{a} \cdot (\mathbf{a}-\mathbf{b}) = \frac{|\mathbf{a}|^2}{2} - \frac{|\mathbf{b}|^2}{2} + \frac{|\mathbf{a}-\mathbf{b}|^2}{2} 
%  \quad \mbox{ for } \; \mathbf{a},\mathbf{b} \in \R^d
%\end{align*}
one easily gets 
\begin{align*}
 \frac{1}{h} \left(\rho \bv - \rho_k \bv_k \right) \cdot \bv = \frac{1}{h} \left( \rho \frac{|\bv|^2}{2} - \rho_k \frac{|\bv_k|^2}{2} \right) 
        + \frac{1}{h} \, (\rho - \rho_k) \, \frac{|\bv|^2}{2} 
        + \frac{1}{h} \rho_k \frac{|\bv-\bv_k|^2}{2} \,.
\end{align*}
Therefore \eqref{equivtimediscretizationline1} with $\bpsi = \bv$ yields 
\begin{align} \label{testline1}
 0 = \int_\Omega \frac{\rho |\bv|^2 - \rho_k |\bv_k|^2}{2h} \, dx
 + \int_\Omega \rho_k \frac{|\bv-\bv_k|^2}{2h} \, dx
 + \int_\Omega 2 \eta(\p_k) |D\bv|^2 \, dx
 + \int_\Omega P_h \p_k \nabla \mu \cdot \bv \, dx \,.
\end{align}
Moreover, multiplying \eqref{timediscretizationline2} with $\mu$ and using the boundary condition for $\mu$, one concludes
\begin{align} \label{testline2}
 0 = \int_\Omega \frac{\p - \p_k}{h} \, \mu \, dx
  + \int_\Omega (\bv \cdot \nabla P_h \p_k) \, \mu \, dx
  + \int_\Omega m(P_h \p_k) |\nabla \mu|^2 \, dx \,.
\end{align}
Furthermore choosing $\psi=\frac{1}{h} (\p - \p_k)$ in  \eqref{timediscretizationline3} we obtain
\begin{align} \label{testline3}
 0 = & \int_\Omega \nabla \p \cdot \nabla (\p - \p_k) \, dx
        + \int_\Omega {\Psi}_0'(\p) \frac{\p - \p_k}{h} \, dx 
        + \frac{1}{h} \E(\p, \p-\p_k)
        \nonumber \\
     & - \int_\Omega \mu \, \frac{\p - \p_k}{h} \, dx
      - \int_\Omega {\kappa} \frac{\p^2 - \p_k^2}{2h} \, dx \,.
\end{align}
Summation of \eqref{testline1}-\eqref{testline3} yields 
\begin{align*}
 0  =& \int_\Omega \frac{\rho |\bv|^2 - \rho_k |\bv_k|^2}{2h} \, dx
     + \int_\Omega \rho_k \frac{|\bv-\bv_k|^2}{2h} \, dx
     + \int_\Omega 2 \eta(\p_k) |D\bv|^2 \, dx
     + \int_\Omega m(P_h \p_k) |\nabla \mu|^2 \, dx \\
   & + \int_\Omega {\Psi}_0'(\p) \frac{\p-\p_k}{h} \, dx
     - \int_\Omega {\kappa} \frac{\p^2 - \p_k^2}{2h} \, dx \\
   & + \int_\Omega \nabla \p \cdot \nabla (\p - \p_k) \, dx 
      + \frac{1}{h} \E(\p, \p - \p_k)
   \\
   \geq & \int_\Omega \frac{\rho |\bv|^2 - \rho_k |\bv_k|^2}{2h} \, dx
     + \int_\Omega \rho_k \frac{|\bv-\bv_k|^2}{2h} \, dx
     + \int_\Omega 2 \eta(\p_k) |D\bv|^2 \, dx
     + \int_\Omega m(P_h \p_k) |\nabla \mu|^2 \, dx \\ 
   & + \frac{1}{h} \int_\Omega \left( {\Psi}_0(\p) - {\Psi}_0(\p_k) \right) \, dx 
     - \int_\Omega \frac{{\kappa}}{2} \, \frac{\p^2 - \p_k^2}{h} \, dx \\
   & + \int_\Omega \frac{|\nabla \p - \nabla \p_k|^2}{2} \, dx
     + \int_\Omega \left( \frac{|\nabla \p|^2}{2} - \frac{|\nabla \p_k|^2}{2} \right) dx \\
   &  + \frac{1}{h} \frac{\E(\p, \p)}{2} - \frac{1}{h}\frac{\E(\p_k, \p_k)}{2} + \frac{1}{h}\frac{\E(\p-\p_k, \p-\p_k)}{2} \,,
\end{align*}
because of $\int_\Omega  P_h \p_k \nabla \mu \cdot \bv \, dx= - \int_\Omega (\bv\cdot \nabla P_h \p_k) \mu   \, dx$,
\begin{align*}
 {\Psi}_0'(\p) \, (\p - \p_k) &\geq {\Psi}_0(\p) - {\Psi}_0(\p_k) \,, \\
 \nabla \p \cdot \nabla (\p - \p_k) &= \frac{|\nabla \p|^2}{2} - \frac{|\nabla \p_k|^2}{2} 
    + \frac{|\nabla \p - \nabla \p_k|^2}{2}\,, \quad \mbox{and} \\
\E(\p, \p - \p_k) &= \frac{\E(\p, \p)}{2} - \frac{\E(\p_k, \p_k)}{2} 
+ \frac{\E(\p - \p_k, \p - \p_k)}{2} \, .    
\end{align*}
This shows \eqref{discreteenergyestimate}. % given by
% \begin{align*} % \label{discreteenergyestimate} (number already defined above)
%   E_{\mbox{\footnotesize tot},h}&(\p,\bv) + \int_\Omega \rho_k \frac{|\bv - \bv_k|^2}{2} \, dx
%     + h \int_\Omega \frac{|\nabla \p - \nabla \p_k|^2}{2} \, dx 
%     + \frac{\E(\p - \p_k, \p - \p_k)}{2}
%     \nonumber \\
%     &+ h \int_\Omega 2 \eta(\p_k) |D\bv|^2 \, dx + h \int_\Omega m(P_h \p_k) |\nabla \mu|^2 \, dx 
%     \leq E_{\mbox{\footnotesize tot}, h}(\p_k,\bv_k) \,. 
        %   \end{align*}

We will prove existence of weak solutions with the aid of the Leray-Schauder principle. 
In order to obtain a suitable reformulation of our time-discrete system we define suitable
$\La_{k},\F_k : X \to Y$, where
\begin{align*}
 X &= \left(H^1_0(\Omega)^d \cap L^2_\sigma(\Omega) \right) \times \D(\partial F_h) \times H^2_n(\Omega) \,, \\
 Y &= \left(H^1_0(\Omega)^d \cap L^2_\sigma(\Omega) \right)' \times L^2(\Omega) \times L^2(\Omega)
\end{align*}
and
\begin{align*}
 \La_k (\bw) = \begin{pmatrix}
              L_k(\bv) \\
              -\mbox{div}(m(P_h \p_k) \nabla \mu) + \int_\Omega \mu \, dx \\
              \p + \partial F_h(\p)
             \end{pmatrix} 
\end{align*}
for every $\bw = (\bv,\p,\mu) \in X$ and 
\begin{align*}
 \left< L_k(\bv),\bpsi \right> &= \int_\Omega 2 \eta(\p_k) D\bv : D\bpsi \, dx \quad \mbox{for all} \;
     \bpsi \in H^1_0(\Omega)^d \cap L^2_\sigma(\Omega). 
%\; \mbox{ and} \\
% \left<-\mbox{div}(m(\p_k) \nabla \mu) + \int_\Omega \mu , \eta \right> 
%    &= \int_\Omega m(\p_k) \nabla \mu \cdot \nabla \eta + \int_\Omega \mu \cdot \int_\Omega \eta \quad \mbox{for} \;
%     \eta \in H^1(\Omega) \,. %\\
% \left< A(\p) + \partial \widetilde{E}(A(\p)) , \eta \right> 
%    &= \int_\Omega A(\p) \, \eta 
%      + \int_\Omega \nabla A(\p) \cdot \nabla \eta
%      + \int_\Omega \widetilde{\Psi}_0'(A(\p)) \, \eta \quad \mbox{for} \; \eta \in H^1(\Omega) \,.
\end{align*}
%Since $\p \in \D(\partial F_h),$ the last line 
%in $\La_k(\bw)$ belongs to $L^2(\Omega)$.
Moreover  we define 
\begin{align*}
 \F_k(\bw) = \begin{pmatrix}
            - \frac{\rho \bv - \rho_k \bv_k}{h} - \mbox{div}(\rho(P_h \p_k) \bv \otimes \bv) - \nabla \mu P_h \p_k 
              - \left(\operatorname{div} \tbj - \frac{\rho - \rho_k}{h} - \bv \cdot \nabla \rho(P_h \p_k) \right) \frac{\bv}{2}
              - \left( \tbj \cdot \nabla \right) \bv  \\
            -\frac{\p - \p_k}{h} - \bv \cdot \nabla P_h \p_k + \int_\Omega \mu \, dx \rule{0cm}{0,5cm} \\
            \p + \mu + \tilde{\kappa} \frac{\p+\p_k}{2} \rule{0cm}{0,6cm}
           \end{pmatrix} 
\end{align*}
for $\bw = (\bv,\p,\mu) \in X$.
By construction $\bw = (\bv,\p,\mu) \in X$ is a solution of \eqref{timediscretizationline1}-\eqref{timediscretizationline3} if and only if
\begin{align*}
 \La_k (\bw) - \F_k(\bw) = 0 \,.
\end{align*}

In \cite[Section~4.2]{ADG13} it is shown that
\begin{equation*}
  L_k \colon H^1_0(\Omega)^d \cap L^2_\sigma(\Omega) \to \left(H^1_0(\Omega)^d \cap L^2_\sigma(\Omega)\right)' 
\end{equation*}
is invertible 
and that for every $f \in L^2(\Omega)$ 
\begin{align} \label{ell1}
 - \mbox{div}(m(P_h \varphi_k) \nabla \mu) + \int_\Omega \mu \, dx = f \, \mbox{ in } \Omega \,, \quad
    \left. \partial_\no \mu \right|_{\partial \Omega} = 0 
\end{align}
has a unique solution $\mu \in H^2_n(\Omega)$. This follows from the Lax-Milgram Theorem and elliptic regularity theory.
Moreover, in  \cite[Section~4.2]{ADG13} the estimate
\begin{align} \label{H2estimatemu}
 \|\mu\|_{H^2(\Omega)} \leq \changes{C_k} \left( \|\mu\|_{H^1(\Omega)} + \|f\|_{L^2(\Omega)} \right) 
\end{align}
is shown.

Because of Theorem~\ref{thm:Regularity}, $\partial F_h$ is maximal monotone and therefore
\begin{align*}
 I + \partial F_h : \mathcal{D}(\partial F_h) \rightarrow L^2(\Omega)
\end{align*}
is invertible. Moreover, $ (I+\partial F_h)^{-1}\colon L^2(\Omega) \to H^{1}(\Omega)$ is continuous, which can be shown as in the proof of Proposition 7.5.5 in~\cite{Abe07}. Since now a nonlocal operator is involved we provide the details for the convenience of the reader. Let $f_l \to_{l\to\infty} f$ in $L^2(\Omega)$ such that 
$f_l = u_l + \partial F(u_l)$ and $f = u + \partial F(u)$ be given. Then 
$u_l \to u$ in $H^1(\Omega)$ since
\begin{align*}
 \|u_l - u\|_{L^2}^2 + h \|\nabla u_l - \nabla u\|_{L^2}^2 + \E(u_l - u, u_l - u) 
   & \leq \|u_l - u\|_{L^2}^2 + \left( \partial F_h(u_l) - \partial F_h(u) , u_l - u \right)_{L^2} \\
   & \leq \|u_l + \partial F_h(u_l) - (u + \partial F_h(u)) \|_{L^2} \, \|u_l - u\|_{L^2} \\
   & \leq \frac{1}{2} \|f_l - f\|_{L^2}^2 + \frac{1}{2} \|u_l - u\|_{L^2}^2 \,.
\end{align*}

Altogether $\La_k : X \to Y$ is invertible with continuous inverse $\La_k^{-1} : Y \to X$.

We introduce the following auxiliary  Banach spaces
\begin{align*}
 \widetilde{X} &\ZuWeis \left( H_0^1(\Omega)^d \cap L_\sigma^2(\Omega) \right) 
                  \times H^{1}(\Omega) \times H^2_n(\Omega) \,, \\
 \widetilde{Y} &\ZuWeis L^{\frac{3}{2}}(\Omega)^d \times W^{1}_{\frac{3}{2}}(\Omega) \times H^1(\Omega) \,
\end{align*}
in order to obtain a completely continuous mapping in the following.
Because of the considerations above $\La_k^{-1} : Y \to \widetilde{X}$ is continuous.
Because of the compact embedding $\widetilde{Y} \hookrightarrow \hookrightarrow Y$, 
$\La_k^{-1} : \widetilde{Y} \to \widetilde{X}$ is compact.

Next we show that $\F_k : \widetilde{X} \to \widetilde{Y}$ is 
continuous and bounded. To this end one uses the estimates:
\begin{align*}
  \|\rho \bv\|_{L^{\frac{3}{2}}(\Omega)} &\leq C \|\bv\|_{H^1(\Omega)} (\|\varphi\|_{L^2(\Omega)} + 1) \,, &
  \|\mbox{div}(\rho(P_h \p_k) \bv \otimes \bv) \|_{L^{\frac{3}{2}}(\Omega)}  &\leq C_k \|\bv\|^2_{H^1(\Omega)} \,, \\
  \|\nabla \mu P_h \p_k \|_{L^{\frac{3}{2}}(\Omega)}  &\leq C_k \|\nabla \mu\|_{L^2(\Omega)} \,, &
  \|(\operatorname{div}\tbj) \bv\|_{L^{\frac{3}{2}}(\Omega)} &\leq C_k \|\bv\|_{H^1(\Omega)} \|\mu\|_{H^2(\Omega)} \,, \\
  \|(\tbj \cdot \nabla) \bv\|_{L^{\frac{3}{2}}(\Omega)} &\leq C \|\bv\|_{H^1(\Omega)} \|\mu\|_{H^2(\Omega)} \,, &
  \|\bv \cdot \nabla \p_k\|_{W^{1}_{\frac{3}{2}}(\Omega)} &\leq C_k \|\bv\|_{H^1(\Omega)} \,.
\end{align*}
 Note that $P_h \varphi_k$ and therefore $\rho(P_h \p_k)$ 
belong to $H^2(\Omega)$).
%\textbf{(Some of the inequalities should be written down in detail, in part. $\mu \in H^2(\Omega) \hookrightarrow C^0(\overline{\Omega})$ 
%in the last inequality)} \\
More precisely:
\begin{enumerate}
 %\item For $\rho \bv$, the estimate is clear due to $\bv \in H^1(\Omega) \hookrightarrow L^6(\Omega)$. %The ``+1'' appears,
       %since $\rho$ depends affine linear on $\varphi$.
 \item For the estimate of $\operatorname{div}(\rho (P_h \p_k) \bv \otimes \bv)$ in $L^{\frac{3}{2}}(\Omega)$, one has to estimate a term of the form $\rho(P_h \p_k) \partial_l \bv_i \bv_j$
       in $L^{\frac{3}{2}}(\Omega)$, which are a product of functions in $L^\infty(\Omega)$, $L^2(\Omega)$ and $L^6(\Omega)$. Therefore the term is bounded in $L^{\frac{3}{2}}(\Omega)$.
       Moreover, there are terms of the form  $\partial_l \rho(P_h \p_k) \bv_i \bv_j$ in $L^{\frac{3}{2}}(\Omega)$, where each factor belongs to $L^6(\Omega)$.
% \item $\nabla \mu P_h \p_k$ in $L^{\frac{3}{2}}(\Omega)$ follows immediately from 
%       $\nabla \mu \in L^2(\Omega)$ and $|P_h \p_k(x)|\leq 1$ almost everywhere.
     \item To estimate $(\operatorname{div}\tbj) \bv$ in $L^{\frac{3}{2}}(\Omega)$ one has terms of the form $m'(P_h \p_k) \partial_i P_h p_k \partial_j \mu \bv_l$ and of the form $m(P_h \p_k) \partial_i \partial_j \mu \bv_l$. For the first type of terms the first factor is in $L^\infty(\Omega)$ and the other three are in $L^6(\Omega)$, which yields the bound in $L^{\frac32}(\Omega)$. The second type are products of  functions
       in $L^\infty(\Omega)$, $L^2(\Omega)$ and $L^6(\Omega)$.
 \item The bound of $(\tbj \cdot \nabla) \bv$ in $L^{\frac{3}{2}}(\Omega)$ follows easily since the factors in $m(P_h \p_k) \partial_i \mu \partial_j \bv_l$ 
       are bounded in $L^\infty(\Omega)$, $L^6(\Omega)$ and $L^2(\Omega)$, respectively.
     \end{enumerate}
     The estimates of the other terms are more easy and left to the reader.
These estimates show the boundedness of $ \F_k $. Using analogous estimates for differences of the terms, one can show the continuity of $ \F_k : \widetilde{X} \to \widetilde{Y} $.

We will now apply the Leray-Schauder principle on $\widetilde{Y}$. To this end we use that  
$\La_k(\bw) - \F_k(\bw) = 0$ for  $\bw \in X$ is equivalent to 
\begin{align} \label{abstractproblem}
 \boldf - \F_k \circ \La_k^{-1}(\boldf) = 0 \quad \mbox{ for }\; \boldf = \La_k(\bw) \,.
\end{align}
Therefore we define $\K_k := \F_k \circ \La_k^{-1} : \widetilde{Y} \to \widetilde{Y}$. We remark that $\K_k$ is 
a compact operator since $\La_k^{-1} : \widetilde{Y} \to \widetilde{X}$ is compact and $\F_k : \widetilde{X} \to \widetilde{Y}$ is continuous. 
Hence \eqref{abstractproblem} is equivalent to the fixed-point equation 
\begin{align*}
% \boldf - \K_k(\boldf) = 0 \quad \Longleftrightarrow \quad 
\boldf = \K_k(\boldf) \,\qquad \text{for }\boldf \in \tilde{Y}.
\end{align*}
%To deduce the existence of such a fixed point with the help of the Leray-Schauder 
    %     principle (see for example Zeidler \cite{Zei92}),
Now we have to show that there is some $R>0$ such that:
\begin{align} \label{requirementLeraySchauder}
 \mbox{ If } \, \boldf \in \widetilde{Y} \, \mbox{ and } \, 0 \leq \lambda \leq 1 
  \, \mbox{ fulfill } \, \boldf = \lambda \K_k \boldf \,, \mbox{ then } \, \|\boldf \|_{\widetilde{Y}} \leq R \,.
\end{align}
To this end we assume that  $\boldf \in \widetilde{Y}$ and $0 \leq \lambda \leq 1$ are such that $\boldf = \lambda \K_k \boldf$. 
Let $\bw = \La_k^{-1}(\boldf)$. Then
\begin{align*}
 \boldf = \lambda \K_k (\boldf) \quad \Longleftrightarrow \quad \La_k(\bw) - \lambda \F_k(\bw) = 0 \,.
\end{align*}
The latter equation is equivalent to 
\begin{align}
 &\int_\Omega 2 \eta(\p_k) D\bv : D\bpsi \, dx
  + \lambda \int_\Omega \frac{\rho \bv - \rho_k \bv_k}{h} \cdot \bpsi \, dx
  + \lambda \int_\Omega \mbox{div}(\rho(P_h \p_k) \bv \otimes \bv) \cdot \bpsi \, dx \nonumber \\
  & \hspace*{10pt} + \lambda \int_\Omega \left( \mbox{div} \tbj - \frac{\rho - \rho_k}{h} - \bv \cdot \nabla \rho(P_h \p_k) \right) \frac{\bv}{2} \cdot \bpsi \, dx
  + \lambda \int_\Omega \left( \tbj \cdot \nabla \right) \bv \cdot \bpsi \, dx\nonumber\\
  & = - \lambda \int_\Omega \nabla \mu P_h \p_k \cdot \bpsi \, dx \label{lambdaproblem1}
\end{align}\
for all $\bpsi \in H_0^1(\Omega)^d \cap L^2_\sigma(\Omega)$ and 
\begin{align}
 & \lambda \frac{\p - \p_k}{h} 
  + \lambda \bv \cdot \nabla P_h \p_k  
  - \lambda \int_\Omega \mu \, dx
  = \mbox{div}(m(P_h \p_k) \nabla \mu)  
  - \int_\Omega \mu \, dx \,, \label{lambdaproblem2} \\
 & \p + \partial F_h(\p) 
  = \lambda \p + \lambda \mu 
                               + \lambda \widetilde{\kappa} \frac{\p + \p_k}{2} \,. \label{lambdaproblem3} 
\end{align}
%First we derive an estimate for $\bw=(\bv,\p,\mu)$ in the norm of $\widetilde{X}$ and additionally an 
%$L^2$-estimate of $\partial F_h(\p)$, then we conclude the desired estimate for $\boldf$ 
%due to the boundedness of $\F_k$. 

{\allowdisplaybreaks
As in the proof of  \eqref{discreteenergyestimate} we choose 
$\bpsi = \bv$ in \eqref{lambdaproblem1}, test \eqref{lambdaproblem2} with $\mu$ and multiply 
\eqref{lambdaproblem3} with $\frac{1}{h}(\p-\p_k)$. In the same way as before one obtains:
\begin{align*}
 0 \; = \; & \lambda \frac{1}{h} \int_\Omega \left( \frac{\rho |\bv|^2}{2} - \frac{\rho_k |\bv_k|^2}{2} \right) 
      + \lambda \frac{1}{h} \int_\Omega \rho_k \frac{|\bv - \bv_k|^2}{2}
      + \int_\Omega 2 \eta(\p_k)  |D\bv|^2 
      + (1-\lambda) \left( \int_\Omega \mu \right)^2 \\
   &  + \int_\Omega m(\p_k) |\nabla \mu|^2 
      + (1-\lambda) \frac{1}{h} \int_\Omega \p( \p-\p_k )
      + \int_\Omega \nabla \p \cdot \left( \nabla \p - \nabla \p_k \right) \\
   &  +\frac1h \E(\p, \p - \p_k)
       + \frac{1}{h} \int_\Omega {\Psi}_0'(\p) (\p - \p_k)
      - \lambda \frac{1}{h} \int_\Omega {\kappa} \frac{\p^2 - \p_k^2}{2} \\
 \; \geq \; & \lambda \frac{1}{h} \int_\Omega \left( \frac{\rho |\bv|^2}{2} - \frac{\rho_k |\bv_k|^2}{2} \right) 
      + \lambda \frac{1}{h} \int_\Omega \rho_k \frac{|\bv - \bv_k|^2}{2}
      + \int_\Omega 2 \eta(\p_k)  |D\bv|^2 
      + (1-\lambda) \left( \int_\Omega \mu \right)^2 \\
   &  + \int_\Omega m(\p_k) |\nabla \mu|^2 
      + (1-\lambda) \frac{1}{h} \int_\Omega \left( \frac{\p^2}{2} - \frac{\p_k^2}{2} \right)
      + \int_\Omega \left( \frac{|\nabla \p|^2}{2} - \frac{|\nabla \p_k|^2}{2} \right) \\
   & +\frac{1}{h} \frac{\E(\p, \p)}{2} - \frac{1}{h}\frac{\E(\p_k, \p_k)}{2} 
   + \frac{1}{h}\frac{\E(\p - \p_k, \p - \p_k)}{2} \\
   &  + \frac{1}{h} \int_\Omega \left( {\Psi}_0(\p) - {\Psi}_0(\p_k) \right)
      - \lambda \frac{1}{h} \int_\Omega {\kappa} \frac{\p^2 - \p_k^2}{2} \,.
\end{align*}}
For brevity we omitted the integration element $dx$.
Thus we obtain
\begin{align*}
 & h \int_\Omega 2 \eta(\p_k) |D\bv|^2 + h \int_\Omega m(\p_k) |\nabla \mu|^2 
 + \frac{h}{2} \int_\Omega |\nabla \p|^2 \\
 & + \int_\Omega {\Psi}(\p) + (1-\lambda)\left(\int_\Omega \mu \, dx \right)^2 + \frac{\E(\p, \p)}{2} \\
 & \leq \int_\Omega \frac{\rho_k |\bv_k|^2}{2} + \frac{1}{2} \int_\Omega \p_k^2 
 + \frac{h}{2} \int_\Omega |\nabla \p_k|^2 + \int_\Omega {\Psi}_0(\p_k)
 + \int_\Omega |{\kappa}| \frac{\p_k^2}{2} + \frac{\E(\p_k, \p_k)}{2} \,.
\end{align*}
Here we used $-\lambda \int_\Omega \widetilde{\kappa} \frac{\p_k^2}{2} \, dx \leq
\lambda \int_\Omega |\widetilde{\kappa}| \frac{\p_k^2}{2} \, dx$ and in addition estimated every 
$\lambda$ resp. $(1-\lambda)$ on the right side by $1$.
% where we omitted the (nonnegative) terms $\lambda \int_\Omega \frac{\rho |\bv|^2}{2} \, dx$, 
% $\lambda \int_\Omega \rho_k \frac{|\bv - \bv_k|^2}{2} \, dx$ and $(1-\lambda) \int_\Omega \frac{\p^2}{2} \, dx$
% on the left side, since due to the factor $\lambda$ resp. $(1-\lambda)$ they will not give 
% a contribution to some estimate of $\|\bw\|_{\widetilde{X}}$ independent of $\lambda$.
Because of $\bw=(\bv,\p,\mu) = \La_k^{-1}(\boldf) \in X$,
 $\p \in \D(\partial F_h)$ and therefore $\p \in [-1,1]$ almost everywhere. In particular we have $\rho \geq 0$. Moreover, 
$\int_\Omega {\Psi}(\p) \, dx$ is bounded.
%We remark that the term $-\lambda \int_\Omega \tilde{\kappa} \frac{A(\p_k)^2}{2}$ is nonnegative
%if we assume w.l.o.g. $\tilde{\kappa} \geq 0$, but this is not necessary. 

Altogether we conclude %We summarize the previous estimate to
\begin{align} \label{apriorilambda}
\nonumber& (1-\lambda)\left(\int_\Omega \mu \, dx \right)^2 + h \int_\Omega 2 \eta(\p_k) |D\bv|^2 \, dx + h \int_\Omega m(\p_k) |\nabla \mu|^2 \, dx \\
& + \frac{h}{2} \int_\Omega |\nabla \p|^2 \, dx + \frac{\E(\p, \p)}{2} \leq C_k \,. 
\end{align}
for some $C_k$ independent of $(\bv, \p,\mu)$.
Using $\|\p\|_{L^\infty} \leq 1$, Korn's inequality, \eqref{eq:EquivNorm2},
and the fact that $\eta$, $m$ and $a$ are bounded
from below by a positive constant, we obtain
\begin{align} \label{apriorilambdashort}
 \sqrt{1-\lambda}\left|\int_\Omega \mu \, dx \right| + \|\bv\|_{H^1(\Omega)} + \|\nabla \mu\|_{L^2(\Omega)} + \|\p\|_{H^1(\Omega)} \leq C_{k} \,.
\end{align}
%where we omit to write the dependence of the constant $C_k$ on $h$,
%since $h$ is fixed. 
In order to estimate $\|\mu\|_{L^2}$, we distinguish the cases $\lambda\in [\frac12,1]$ and $\lambda\in [0,\frac12)$. In the case 
$\lambda \in [\frac{1}{2},1]$, we simply use
$\frac{1}{2} | \int_\Omega \mu \, dx| \leq \lambda |\int_\Omega \mu \, dx|$
and conclude as in the proof of Lemma \ref{derivativeforchempot} together with \eqref{apriorilambdashort} from 
\eqref{lambdaproblem3} that
\begin{align*}
 \left| \int_\Omega \mu \, dx \right| \leq C_k \,.
\end{align*}
In the case $\lambda \in [0,\frac{1}{2})$ we conclude directly from \eqref{apriorilambdashort}  that
$\left| \int_\Omega \mu \, dx \right| \leq C_k$.
%For $\lambda \in [0,\frac{1}{2})$ we first integrate \eqref{lambdaproblem2} to get 
%$(\lambda - 1) |\Omega| \int_\Omega \mu \, dx = \lambda \int_\Omega \frac{\rho - \rho_k}{h} \, dx
%+ \lambda \int_\Omega \bv \cdot \nabla \varphi_k \, dx$. Then we use 
%$\frac{1}{2} |\int_\Omega \mu \, dx| \leq (1-\lambda) |\int_\Omega \mu \, dx|$ together with 
%\eqref{apriorilambdashort} to see also here $\left| \int_\Omega \mu \, dx \right| \leq C_k$. 
Thus \eqref{apriorilambdashort} can be improved to
\begin{align} \label{apriorilambdashortfinal}
 \|\bv\|_{H^1(\Omega)} + \|\mu\|_{H^1(\Omega)} + \|\p\|_{H^1(\Omega)} \leq C_k \,.
\end{align}
With the help of \eqref{H2estimatemu} we can estimate $\|\mu\|_{H^2(\Omega)}$ and derive
\begin{align} \label{apriorilambdashortfinally}
 \|\bv\|_{H^1(\Omega)} + \|\mu\|_{H^2(\Omega)} + \|\p\|_{H^1(\Omega)} \leq C_k \,.
\end{align}
Using \eqref{lambdaproblem3} we also have $\| \partial F_h(\p) \|_{L^2(\Omega)} \leq C_k$.
Altogether we conclude
\begin{align*}
 \|\bw\|_{\widetilde{X}} + \|\partial F_h(\p)\|_{L^2(\Omega)} 
 = \|(\bv,\p,\mu)\|_{\widetilde{X}} + \|\partial F_h(\p)\|_{L^2(\Omega)} \leq C_k \,.
\end{align*}
Finally we can estimate $\boldf = \La_k(\bw)$ in $\widetilde{Y}$ by using that
$\boldf - \lambda \F_k \La_k^{-1} (\boldf) = 0$ implies $\boldf = \lambda \F_k(\bw)$ together with the boundedness of
$\F_k : \widetilde{X} \to \widetilde{Y}$. Thus we obtain
\begin{align*}
 \|\boldf\|_{\widetilde{Y}} = \|\lambda \F_k(\bw) \|_{\widetilde{Y}} \leq C'_k \,.
\end{align*}
Thus the condition of the Leray-Schauder principle is satisfied, which proves the existence of a solution.
\end{proof}

%%% Weiter!!!
\section{Proof of Theorem~\ref{existenceweaksolution}}\label{sec:proof}
\subsection{Compactness in Time}

In order to prove our main result Theorem \ref{existenceweaksolution} we send $h \to 0$ 
resp. $N \to \infty$ for the approximate solution, which are obtain by suitable interpolations of our time-discrete solutions. To this end let $N \in \mathbb{N}$ be given and let 
$(\bv_{k+1},\p_{k+1},\mu_{k+1})$, $k\in \N$, be chosen successively as a solution of 
\eqref{timediscretizationline1}-\eqref{timediscretizationline3} with $h = \frac{1}{N}$ and $(\bv_0,\varphi_0^N)$ \changes{where $ \varphi_0^N = P_h \varphi_0 $} as initial value. 

As in \cite{ADG13}  we define $f^N(t)$ for $t\in[-h,\infty)$ by the relation $f^N(t) = f_k$ for $t \in [(k-1)h,kh)$, where
$k \in \mathbb{N}_0$ and $f \in \{\bv,\p,\mu\}$. 
Moreover, let $\rho^N = \frac{1}{2}(\tilde{\rho}_1 + \tilde{\rho}_2) + \frac{1}{2}(\tilde{\rho}_2 - \tilde{\rho}_1) \varphi^N$.
%In particular it holds $f^N((k-1)h) = f_k$, $f^N(kh) = f_{k+1}$ and $f^N(t) = f_{k+1}$ for $t \in [kh,(k+1)h)$. 
Furthermore we introduce the notation
\begin{align*}
 \left( \Delta^+_h f \right)(t) &:= f(t+h)-f(t) \,,&   \left( \Delta^-_h f \right)(t) &:= f(t)-f(t-h) \,, \\
 \partial^\pm_{t,h} f(t) &:= \frac{1}{h} \left( \Delta_h^\pm f \right)(t) \,, & 
 f_h &:= \left( \tau_h^\ast f\right)(t) = f(t-h) \,. 
\end{align*}
In order to derive the weak formulation in the limit let $\bpsi \in \left( C^\infty_0(\Omega \times (0,\infty)) \right)^d$ with $\operatorname{div} \bpsi = 0$ be arbitrary and choose
$\widetilde{\bpsi} := \int_{kh}^{(k+1)h} \bpsi \, dt$ as test function in 
\eqref{timediscretizationline1}. By summation with respect to $k \in \mathbb{N}_0$ this yields
\begin{align} 
  \int_0^\infty &\hspace*{-5pt} \int_\Omega \partial^-_{t,h}(\rho^N \bv^N) \cdot \bpsi \, dx \,dt
  + \int_0^\infty \hspace*{-5pt} \int_\Omega \mbox{div} \left(\rho^N_h \bv^N \otimes \bv^N \right) \cdot \bpsi \, dx \,dt
  + \int_0^\infty \hspace*{-5pt} \int_\Omega 2 \eta(\p^N_h) D\bv^N : D\bpsi \, dx \,dt \nonumber \\
 & - \int_0^\infty \hspace*{-5pt} \int_\Omega \left( \bv^N \otimes \tbj^N \right) : D\bpsi \, dx \,dt
   = - \int_0^\infty \hspace*{-5pt} \int_\Omega \nabla \mu^N \p^N_h \cdot \bpsi \, dx \,dt \label{timeintegratedline1}
\end{align}
for all $\bpsi \in \left( C^\infty_0(\Omega \times (0,\infty)) \right)^d$ with $\operatorname{div} \bpsi = 0$.
\changes{Here $ \rho^N_h = (\rho^N)_h $ and $ \varphi^N_h = (\varphi^N)_h. $}
Using a simple change of variable, one sees
\begin{align*}
 \int_0^\infty \hspace*{-5pt} \int_\Omega \partial^-_{t,h}(\rho^N \bv^N) \cdot \bpsi \, dx \,dt
   = - \int_0^\infty \hspace*{-5pt} \int_\Omega (\rho^N \bv^N) \cdot \partial^+_{t,h}\bpsi \, dx \,dt \,
\end{align*}
for sufficiently small $h>0$.
In the same way one derives
\begin{align} \label{timeintegratedline2}
 \int_0^\infty \hspace*{-5pt} \int_\Omega \partial^-_{t,h} \p^N \, \zeta \, dx \,dt
  + \int_0^\infty \hspace*{-5pt} \int_{\Omega} \bv^N \p^N_h \cdot \nabla \zeta \, dx \,dt
  = \int_0^\infty \hspace*{-5pt} \int_\Omega m(\p^N_h) \nabla \mu^N \cdot \nabla \zeta \, dx \,dt
\end{align}
for all $\zeta \in C^\infty_0((0,\infty);C^1(\overline{\Omega}))$ as well as
\begin{align} \label{timeintegratedline3}
 \int_{0}^{\infty} 
 \int_\Omega (\mu^N + {\kappa} \, \frac{\p^N + \p_h^N}{2})\psi \, dx \, dt 
 = & \int_{0}^{\infty} \E(\p^N,\psi) \,dt + \int_{0}^{\infty} \int_\Omega {\Psi}_0'(\p^N)\psi\, dx \, dt \nonumber\\
& +h \int_{0}^{\infty} \int_\Omega \nabla \p^N \cdot \nabla \psi \, dx \, dt
\end{align}
for all $\psi\in C^\infty_0((0,\infty);C^1(\overline{\Omega}))$.
%
%\begin{align} 
% \frac{\Delta^-_h \p^N}{\Delta^-_h A(\p^N)} \, \mu^N 
%  + \frac{\widetilde{\kappa}}{2} \left( A(\p^N) + A(\p^N_h) \right) 
% &= - h \Delta A(\p^N) + \widetilde{\Psi}_0'(A(\p^N)) + \LL\p^N \label{timeintegratedline3}
%\end{align}
%holds pointwise in $\Omega \times (0,\infty)$ almost everywhere.

Let $E^N(t)$ be defined as %the piecewise linear interpolant of $E_{\mbox{\footnotesize tot}}(\p_k,\bv_k)$ at
%$t_k = kh$, i.e.,
\begin{align*}
 E^N(t) = \frac{(k+1)h - t}{h} E_{\mbox{\footnotesize tot}}(\p_k,\bv_k) 
         + \frac{t - kh}{h} E_{\mbox{\footnotesize tot}}(\p_{k+1},\bv_{k+1}) \; \mbox{ for } \; t \in [kh,(k+1)h) 
\end{align*}
and set 
\begin{align*}
 D^N(t) := \int_\Omega 2 \eta(\p_k) |D\bv_{k+1}|^2 \, dx + \int_\Omega m(\p_k) |\nabla \mu_{k+1}|^2 \, dx 
\end{align*}
for all $t \in (t_k,t_{k+1})$, $k \in \mathbb{N}_0$.
Then  \eqref{discreteenergyestimate} yields
\begin{align} \label{inequalityforEandD}
 - \frac{d}{d t} E^N(t) = \frac{E_{\mbox{\footnotesize tot}}(\p_k,\bv_k) - E_{\mbox{\footnotesize tot}}(\p_{k+1},\bv_{k+1})}{h}
    \geq D^N(t) 
\end{align}
for all $t \in (t_k,t_{k+1})$, $k \in \mathbb{N}_0$. % Multiplying this inequality by 
% $\tau \in W^{1,1}(0,\infty)$ with $\tau \geq 0$, integrating and using integration by parts yields
% \begin{align} \label{integratedinequEandD}
%  E_{\mbox{\footnotesize tot}}(\p_0^N,\bv_0) \tau(0) + \int_0^\infty E^N(t) \, \tau'(t) \, dt \geq \int_0^\infty D^N(t) \, \tau(t) \, dt \,.
% \end{align}
Integration implies %\eqref{inequalityforEandD}
%with respect to $t$ gives
\begin{align} %\label{inequforEN}
 E_{\mbox{\footnotesize tot}}(\p^N(t),\bv^N(t)) 
  &+ \int_s^t \int_\Omega \left( 2 \eta(\p^N_h) |D\bv^N|^2 + m(\p^N_h) |\nabla \mu^N|^2 \right) dx \,d\tau \nonumber \\
  &\leq E_{\mbox{\footnotesize tot}}(\p^N(s),\bv^N(s)) \label{inequforEN}
\end{align}
for all $0 \leq s \leq t < \infty$ with $s,t \in h\mathbb{N}_0$. 

Because of Lemma \ref{derivativeforchempot} and since $E_{\mbox{\footnotesize tot}}(\varphi_0^N,\bv_0)$
is bounded, we conclude that 
\begin{align} \label{timediscrbounds}
 \begin{array}{l}
  (\bv^N)_{N\in \N} \subseteq L^2(0,\infty;H^1(\Omega)^d)\cap L^\infty(0,\infty; L^2(\Omega)^d) \,, \\
  (\nabla \mu^N)_{N\in\N} \subseteq L^2(0,\infty;L^2(\Omega)^d) \,, \rule{0cm}{0,5cm}\\
  (\varphi^N)_{N\in\N} \subseteq L^\infty(0,\infty;H^{\frac{\alpha}{2}}(\Omega)) \,, \mbox{ and } \rule{0cm}{0,5cm} \\
   (h^{\frac12}\nabla \p^N)_{N\in\N} \subseteq L^{\infty}(0, \infty;L^2(\Omega)) 
 \end{array}
\end{align}
are bounded. Moreover, there is
a nondecreasing $C\colon (0,\infty) \to (0,\infty)$ such that
\begin{equation*}
  \int_0^T \left| \int_\Omega \mu^N \, dx \right| dt \leq C(T) \, \mbox{ for all } \, 0<T<\infty\,. \rule{0cm}{0,5cm}
\end{equation*}
Therefore there are subsequences (denoted again by the index $N\in\N$, $h>0$, respectively) such that
\begin{align*}
  \bv^N \rightharpoonup \bv \; &\mbox{ in } \, L^2(0,\infty;H^1(\Omega)^d) \,, \\
  \bv^N \rightharpoonup^\ast \bv \; &\mbox{ in } \, L^\infty(0,\infty;L^2(\Omega)^d) 
  % \cong \left( L^1(0,\infty;L^2(\Omega)^d) \right)'
                                      \,, \\
  \varphi^N \rightharpoonup^\ast \varphi \; &\mbox{ in } \, L^\infty(0,\infty;H^{\frac{\alpha}{2}}(\Omega)) 
  % \cong \left( L^1(0,\infty;H^{\frac{\alpha}{2}}(\Omega)) \right)'
                                              \,, \\
  \mu^N \rightharpoonup \mu \; &\mbox{ in } \, L^2(0,T;H^1(\Omega)) \, \mbox{ for all } \, 0<T<\infty \,, \\
  \nabla \mu^N \rightharpoonup \nabla \mu \; &\mbox{ in } \, L^2(0,\infty;L^2(\Omega)^d) \,,
\end{align*}
where $\mu\in L^2_{uloc}([0,\infty);H^1(\Omega))$.
%Here and in the following all limits are meant to be for suitable subsequences 
%$N_k \to \infty$ (resp. $h_k \to 0$) for $k \to \infty$. 

In the following $\widetilde{\varphi}^N$ denotes the piecewise linear interpolant of $\varphi^N(t_k)$ in time, where 
$t_k = kh$, $k \in \mathbb{N}_0$. %Thus $\widetilde{\varphi}^N = \frac{1}{h} \chi_{[0,h]} \ast_t \varphi^N$,
%where the convolution is only taken with respect to the time variable $t$. 
Then  $\partial_t \widetilde{\varphi}^N = \partial_{t,h}^- \varphi^N$ and therefore
\begin{align} \label{estofdiffint}
 \|\widetilde{\varphi}^N - \varphi^N \|_{H^{-1}(\Omega)} \leq h \|\partial_t \widetilde{\varphi}^N\|_{H^{-1}(\Omega)} \,.
\end{align}
Using that $\bv^N \varphi^N$ and $\nabla \mu^N$ are bounded in $L^2(0,\infty; L^2(\Omega)^d)$ and \eqref{timeintegratedline2} we conclude that $\partial_t \widetilde{\varphi}^N \in L^2(0,\infty; \changes{H_{(0)}^{-1}}(\Omega))$
is bounded. 
%(Even $v^N \varphi^N \in L^2(0,\infty; L^2(\Omega))$ would be enough). 
Since  $(\varphi^N)_{N\in\N}$ and therefore $(\widetilde{\varphi}^N)_{N\in\N}$ are bounded in $L^\infty(0,\infty;H^{\frac{\alpha}{2}}(\Omega))$, the Lemma of Aubin-Lions yields  
\begin{align}\label{p-convergence1}
 \widetilde{\varphi}^N \to \widetilde{\varphi} \; \mbox{ in } \; L^2(0,T;L^2(\Omega)) 
\end{align}
for all $0<T<\infty$ for some $\widetilde{\varphi} \in L^\infty(0,\infty;L^2(\Omega))$ (and a suitable subsequence). 
In particular $\widetilde{\varphi}^N(x,t) \to \widetilde{\varphi}(x,t)$
 almost every $ (x,t)\in \Omega\times (0,\infty)$. 
Because of  \eqref{estofdiffint},
\begin{align}\label{p-convergence2}
\| \widetilde{\varphi}^N - \varphi^N \|_{L^2(-h, \infty;H^{-1}(\Omega))} \to 0 
\end{align}
and thus $\widetilde{\varphi} = \varphi$.
Since $\widetilde{\varphi}^N \in H^1_{\mbox{\footnotesize uloc}}([0,\infty);H^{-1}(\Omega)) 
\cap L^{\infty}([0,\infty);H^{\frac{\alpha}{2}}(\Omega))
\hookrightarrow BUC([0,\infty);L^2(\Omega))$ and $\widetilde{\varphi}^N \in L^\infty(0,\infty;H^{\frac{\alpha}{2}}(\Omega))$
are bounded, Lemma \ref{lem:CwEmbedding} implies $\varphi \in BC_w([0,\infty);H^{\frac{\alpha}{2}}(\Omega))$. 
Moreover, $ (\widetilde{\p}^N - {\p}^N)_{N\in\N} \subseteq L^\infty(-h,\infty;H^{\frac{\alpha}{2}}(\Omega))$ is bounded
since $ ({\p}^N)_{N\in\N}, (\widetilde{\p}^N)_{N\in\N} \subseteq  L^\infty(-h,\infty;H^{\frac{\alpha}{2}}(\Omega)) $ are bounded. By  interpolation with
\eqref{p-convergence2} we conclude 
\begin{align}\label{p-convergence3}
\widetilde{\p}^N - \p^N \to 0 \; \mbox{in} \; L^2(-h, T;L^2(\Omega)) \,
\end{align}
and therefore %. $ Using $ \widetilde{\p} = \p, $ \eqref{p-convergence1}  and \eqref{p-convergence3}, we have
\begin{align}\label{p-convergence4}
{\p}^N \to \p \; \mbox{in} \; L^2(0, T;L^2(\Omega)) \,
\end{align}
for all $ 0<T<\infty. $   
Moreover, we have
\begin{align}\label{p-convergence5}
&\| \p^N_h - \p \|_{L^2(0, T; L^2(\Omega))} \leq \|\p^N_h - \p_h\|_{L^2(0, T; L^2(\Omega))} + \| \p_h - \p \|_{L^2(0, T;L^2(\Omega))} \nonumber\\
&\leq h^{\frac12}\|\p_0^N\|_{L^2(\Omega)}+\|\p^N - \p \|_{L^2(0, T-h; L^2(\Omega))}+ \| \p_h - \p \|_{L^2(0, T;L^2(\Omega))}.
\end{align}
Because of $\| \p_h - \p \|_{L^2(0, T;L^2(\Omega))}\to_{h\to 0} 0$, 
we obtain $ \| \p^N_h - \p \|_{L^2(0, T; L^2(\Omega))} \to_{h\to 0} 0$.

%(\textbf{The spaces and the Lemma have to be stated precisely in the introduction.})

Finally using the bounds of $\widetilde{\varphi}^N$ in $H^1(0,T;H^{-1}(\Omega))\cap L^\infty(0,T;\changes{H^{\frac{\alpha}{2}}}(\Omega))$ 
for all $0<T<\infty$ as well as $\widetilde{\varphi}^N \to \varphi$ in $L^2(0,T;L^2(\Omega))$ we conclude 
$\widetilde{\varphi}^N(0) \to \varphi(0)$ in $L^2(\Omega)$. Since $\widetilde{\varphi}^N(0)=\varphi_0^N\to_{N\to \infty}\varphi_0$ in $L^2(\Omega)$, we derive $\varphi(0) = \varphi_0$.

Since $\rho^N$  depends affine linearly on $\varphi^N$, the conclusions hold true for $\rho^N$.

To pass to the limit in \eqref{timeintegratedline3}, 
we closely follow the corresponding argument in \cite{ABG15}. The only difference is that \changes{we} work on the space-time domains directly, while they work on the spacial domains fixing a time variable in \cite{ABG15}. We include the argument here for completeness.
\changes{We first observe that 
$ \Psi_0'(\p^N)$ are bounded 
in $L^2_{uloc}([0, \infty);L^2(\Omega))$ using Lemma  
\ref{derivativeforchempot} and the boundedness of $ \nabla  \mu_N $
in $ L^{2}(0, \infty; L^2(\Omega)).$}
Using this bound, we can pass to
a subsequence such that $ \Psi_0'(\p^N) $ converges weakly in   $L^2(0,T;L^2(\Omega))$ to $\chi$ for all $0<T<\infty$ as $N$ tends to infinity. 
Let  $\psi\in C^\infty_0((0,\infty);C^1(\overline{\Omega}))$.
Thanks to the convergences listed above, we can pass to the limit  $N \rightarrow \infty$ in \eqref{timeintegratedline3} to find
\begin{align*}
\int_{0}^{\infty} \int_{\Omega} (\mu + \kappa \p) \psi \,dx \,dt
= \int_{0}^{\infty} \E(\p, \psi) \,dt + (\chi, \psi)_{L^2((0, \infty) \times \Omega)}.
\end{align*} 
To show \eqref{weakline3}, we only have to identify the weak limit
$ \chi = \lim_{N \rightarrow \infty} \Psi_0'(\p^N) $.
Let $ T>0 $. Since \eqref{p-convergence4} holds, passing to a subsequence,
we have $ \p^N \rightarrow \p $ almost everywhere in $ \Omega \times (0, T) $. 
On the other hand, thanks to Egorov's theorem, there exists a set $Q_m
 \subset \Omega \times (0, T)$ such that $ |Q_m| \geq |\Omega \times (0, T)| - \frac{1}{2m} $ and on which $ \p^N \rightarrow \p$ uniformly. We now use (uniform with respect to $N$) estimate on $\Psi_0'(\p^N)$ in $L^2(\Omega \times (0, T))$.
By definition, the quantity
 \begin{align*}
 M_{\delta, N} = \left| \left\{ (x, t) \in \Omega \times (0, T) \mid |\p^N(x, t)| > 1- \delta \right\} \right|
 \end{align*} 
is decreasing in $\delta$ for all $ n \in \N. $   
Since $\Psi_0'(y)$ is unbounded for $ y \rightarrow \pm 1$, we set
$$ c_{\delta} := \inf_{|c| \geq 1- \delta} |\Psi_0'(c)| \rightarrow_{\delta \rightarrow 0} \infty, $$ 
we have by the Tschebychev inequality
\begin{align*}
\int_{\Omega \times (0, T)} | \Psi_0'(\p^N)|^2\,dx\,dt \geq c_{\delta}^2 |M_{\delta, N}|.
\end{align*}
From the uniform (with respect to $N$) estimate of the norm of $ \Psi_0'(\p^N) $ in $ L^2(\Omega \times (0, T)), $ we obtain $ M_{\delta, n} 
\rightarrow 0 $ for $ \delta \rightarrow 0 $ uniformly in $ n \in \N. $
Therefore, we deduce 
\begin{align*}
\lim_{\delta \rightarrow 0} \left| \left\{ (x, t) \in \Omega \times (0, T) \mid |\p^N(x, t)| > 1- \delta \right\} \right|=0
\end{align*}
uniformly in $ N \in \N. $
Thus there exists $ \delta = \delta(m) $ independent of $N$, such that
\begin{align*}
\left| \left\{ (x, t) \in \Omega \times (0, T) \mid | \p^N(x, t) | > 1- \delta \right\} \right| \leq \frac{1}{2m},~~~~\forall N \in \N
\end{align*}
Consider now $ N \in \N $ so large that by uniform convergence we have $ |\p^{N'}(x,t) - \p^N(x,t)| < \frac{\delta}{2}$ for all $N' \geq N$ and all $(x,t)\in  Q_m $. Moreover, let $ Q_{mN}' \subset Q_m $ be defined by  
\begin{align*}
Q_{mN}' = Q_m \cap \left\{ (x, t) \in \Omega \times (0, T) \mid 
| \p^{N}(x, t) | \leq 1 - \delta \right\}.
\end{align*}
By the above construction, we immediately deduce that $ | Q_{mN}' | \geq |\Omega \times (0, T) | - \frac{1}{m} $ and that $ \left| \p^{N'}(x, t) \right| < 1- \frac{\delta}{2} $ for all $ N' \geq N $ and for all $ (x,t) \in Q_{m,N}. $
Therefore by the regularity assumptions on the potential $ \Psi_0', $ we deduce 
that $ \Psi_0'(\p^N) \rightarrow \Psi_0'(\p) $ uniformly on $ Q_{mN}'. $
Since $m $ is arbitrary, we have  $ \Psi_0'(\p^N) \rightarrow \Psi_0'(\p) $ almost everywhere in $ \Omega \times (0, T). $ By a diagonal argument, 
passing to a subsequence, we have  $ \Psi_0'(\p^N) \rightarrow \Psi_0'(\p) $
almost everywhere in $ \Omega \times (0, \infty) $ and $\Psi_0'(\p^N) \rightarrow \Psi_0'(\p)$ as $h\to 0$ in $L^q(Q_T)$ for every $1\leq q<2$ and $0<T<\infty$. Finally, the uniqueness of weak and strong limits gives $ \chi = \Psi_0'(\p) $ as claimed.

Next we show $\bv^N \to \bv$ in $L^2(0,T;L^2(\Omega)^d)$ for all
$0<T<\infty$ and almost everywhere.
 We note that 
 $\partial_t \left(\widetilde{\rho \bv}^N\right) = \partial_{t,h}^- \left(\rho^N \bv^N\right)$ since $\widetilde{\rho \bv}^N$ is the piecewise linear interpolant of $\left(\rho^N \bv^N\right)(t_k)$.
Using that 
\begin{align*}
 \rho^N_h \bv^N \otimes \bv^N & \; \mbox{ is bounded in } \; L^2(0,T;L^\frac{3}{2}(\Omega)) \,, \\
 D\bv^N                       & \; \mbox{ is bounded in } \; L^{2}(0,T;L^2(\Omega)) \,, \\
 \bv^N \otimes \nabla \mu^N          & \; \mbox{ is bounded in } \; L^{\frac{8}{7}}(0,T;L^{\frac{4}{3}}(\Omega)) \,, \\
 \nabla \mu^N \varphi^N_h     & \; \mbox{ is bounded in } \; \changes{L^2(0,T;L^2(\Omega))} \,.
\end{align*}
together with \eqref{timeintegratedline1}, we obtain that 
$\partial_t \left( \mathbb{P}_\sigma(\widetilde{\rho \bv}^N) \right)$ is bounded in 
$L^{\frac{8}{7}}(0,T; (W^1_6(\Omega))')$ for all $0<T<\infty$. 
Here we remark that the boundedness of $\nabla \mu^N \in L^2(0,T;L^2(\Omega))$ and $\varphi^N_h \in \changes{L^\infty(0,T;L^{\infty}(\Omega))}$ imply that $\nabla \mu^N\varphi^N_h \in \changes{L^2(0,T;L^{2}(\Omega))} $ is bounded.
% In detail, these bounds are derived as follows, where for abbreviation ``$\in$'' always means bounded independent of $N$ in the 
% corresponding space and we omit the time and space variables. 
% \begin{enumerate}
%  \item Due to $\bv^N \in L^\infty(L^2) \cap L^2(L^6)$, we get for products
%        $\bv^N_i \bv^N_j \in L^2(L^{\frac{3}{2}})$ and from $\rho^N_h \in L^\infty(\Omega_T)$ this also holds 
%        for $\rho^N_h \bv^N \otimes \bv^N$.
%  \item From $\bv \in L^2(H^1)$ we immediately get $D\bv^N \in L^2(L^2)$. 
%  \item Due to $\bv^N \in L^\infty(L^2) \cap L^2(L^6)$ and $\nabla \mu^N \in L^2(L^2)$ we get $\bv^N_i \partial_j \mu^N \in L^2(L^1)
%        \cap L^1(L^{\frac{3}{2}}) \hookrightarrow L^{\frac{8}{7}}$. 
%  \item From $\nabla \mu^N \in L^2(L^2)$ and $\varphi^N_h \in L^\infty(L^{\frac{6}{3-\alpha}})$ we get finally 
%        $\nabla \mu^N\varphi^N_h \in L^2(L^{\frac{6}{6 - \alpha}}) \hookrightarrow L^2(L^{\frac65})$. 
% \end{enumerate}
%This means in particular that all four terms are bounded in $L^{\frac{8}{7}}(0,T;L^{\frac{6}{5}}(\Omega))$.
%Therefore we can choose in \eqref{timeintegratedline1}  test functions 
%$\bpsi \in 
%L^8(0,T;W^{1}_{6}(\Omega))$. This implies the bound for $\partial_t \left( \mathbb{P}_\sigma(\widetilde{\rho \bv}^N) \right)$
%in $\left(L^8(0,T;W^{1}_{6}(\Omega))\right)' = L^\frac{8}{7}(0,T;(W^1_6(\Omega))')$.

\changes{Since $ \rho^N $ is bounded in $ L^{\infty}(0, T; H^{\frac{\alpha}{2}}(\Omega)^d) $ and $ \bv^N $ is bounded in $ L^2(0, T; H^{1}(\Omega)^d) $,
  using a product rule for Besov spaces, cf.}~\cite{RS96}\changes{,
  suitable Sobolev embeddings and the boundedness of $\mathbb{P}_{\sigma}$ in Sobolev spaces, we have the boundedness of 
$  \mathbb{P}_\sigma (\widetilde{\rho \bv}^N) $  in $ L^2(0,T;H^{\epsilon}(\Omega)^d) $ for some $ \epsilon > 0.$

Hence} the Lemma of Aubin-Lions implies
\begin{align*}
 \mathbb{P}_\sigma (\widetilde{\rho \bv}^N) \to \bw
  \; \mbox{ in } \; L^2(0,T;L^2(\Omega)^d)
\end{align*}
for all $0<T<\infty$ for some $\bw \in L^\infty(0,\infty;L^2(\Omega)^d)$. 
Since the projection $\mathbb{P}_\sigma : L^2(0,T;L^2(\Omega)^d) \to L^2(0,T;L^2_\sigma(\Omega))$ is weakly continuous, we 
conclude from the weak convergence $\widetilde{\rho \bv}^N \rightharpoonup \rho \bv$ 
in $L^2(0,T;L^2(\Omega))$ that $\bw = \mathbb{P}_\sigma (\rho \bv)$.
%Additionally by the above estimate \eqref{estofdiffint2} we have 
%\begin{align*}
% \mathbb{P}_\sigma \left( \widetilde{\rho \bv}^N - \rho^N \bv^N \right) \to 0 \; \mbox{ in } \; L^2(0,T;H^{-1}(\Omega)) \,,
%\end{align*}
%which gives $\widetilde{\rho \bv} = \rho \bv$. 
%Now we derive 
                                %                                 $\bv^N \to \bv$ in $L^2(0,T;L^2(\Omega)^d)$ with the help of the following observations:
This yields
\begin{align*}
 \int_0^T \int_\Omega \rho^N |\bv^N|^2 = \int_0^T \int_\Omega \mathbb{P}_\sigma (\rho^N \bv^N ) \cdot \bv^N
 \longrightarrow \int_0^T \int_\Omega \mathbb{P}_\sigma(\rho \bv) \cdot \bv 
  = \int_0^T \int_\Omega \rho \, |\bv|^2 \,
\end{align*}
because of $\mathbb{P}_\sigma(\rho^N \bv^N)\to_{N\to\infty } \mathbb{P}_\sigma (\rho \bv)$ in $L^2(0,T;L^2(\Omega)^d)$.
Since weak convergence and convergence of the norms imply strong convergence in a Hilbert space, we conclude $(\rho^N)^{\frac{1}{2}} \bv^N \to (\rho)^{\frac{1}{2}} \bv$ in $L^2(0,T;L^2(\Omega)^d)$.
Because of 
\begin{align*}
 \rho^N \to \rho \; \mbox{ almost everywhere in } \, (0,\infty) \times \Omega \; \mbox{ and } \;
  |\rho^N| \geq c > 0 \,,
\end{align*}
we derive
\begin{align*}
 \bv^N = (\rho^N)^{-\frac{1}{2}} \left( (\rho^N)^{\frac{1}{2}} \bv^N \right)
  \to_{N\to\infty} \bv \; \mbox{ in } \; L^2(0,T;L^2(\Omega)^d) \,.
\end{align*}
This yields  $\bv^N \to_{N\to\infty} \bv$ almost everywhere in $(0,\infty) \times \Omega$
(for a subsequence). 

Now we can pass to the limit in \eqref{timeintegratedline1},
\eqref{timeintegratedline2} to get \eqref{weakline1}, \eqref{weakline2} with the aid of the previous results using
that for all divergence free $\bpsi$
\begin{align*}
 \int_0^T \int_\Omega \nabla \mu^N P_N \varphi^N_h \cdot \bpsi \, dx\, dt \to_{N\to \infty} 
  \int_0^T \int_\Omega \nabla \mu \varphi \cdot \bpsi \, dx\, dt\,.
\end{align*}
The initial condition $ \bv(0) = \bv_0 $ in $ L^2(\Omega)^d$ is shown in the same
way as in \cite{ADG13}. Therefore we omit the proof.

\changes{Finally, using \eqref{eq:4}, $\Psi'(\varphi)\in L^2_{uloc}([0,\infty); L^2(\Omega))$ and the local regularity result due to \cite[Lemma~4.3]{AK07} we obtain $\varphi \in L^2_{uloc}([0,\infty);H^{\alpha}(\Omega'))$ for every open $\Omega'$ with $\overline{\Omega'}\subseteq \Omega$, i.e., $\varphi \in L^2_{uloc}([0,\infty);H^{\alpha}_{loc}(\Omega))$.}

\subsection{Proof of the Energy Inequality}

It remains to show the energy inequality \eqref{weakline5}.
If we show that $ \p^N(t) \to \p(t) $ in $ H^{\frac{\alpha}{2}}_{(m)} $ for 
almost every $ t \in (0, \infty) $ and $ \sqrt{h} \nabla \p^N \to 0 $ in 
$ (L^2(\Omega))^d $ for almost every $ t \in (0, \infty), $ the rest of the proof is
almost the same as in \cite{ADG13} and we omit it.
To this end it suffices to show 
$ (\p^N, \sqrt{h} \nabla \p^N) $ converges strongly to $ (\p, 0) $ in $L^2(0, T; H^{\frac{\alpha}{2}}_{(m)}(\Omega) \times (L^2(\Omega))^d) $ for every $T>0$.
If we take $ \psi = \p^N $ in \eqref{timeintegratedline3} \changes{(after a standard approximation)}, we have
\begin{align} 
 \int_{0}^{\infty} 
 \int_\Omega \left(\mu^N + {\kappa} \, \frac{\p^N + \p_h^N}{2}\right)\p^N \, dx \, dt 
 = & \int_{0}^{\infty} \E(\p^N,\p^N) \,dt + \int_{0}^{\infty} \int_\Omega {\Psi}_0'(\p^N)\p^N\, dx \, dt \nonumber\\
& +h \int_{0}^{\infty} \int_\Omega \nabla \p^N \cdot \nabla \p^N \, dx \, dt\,.
\end{align}
Since $ \p^N \to \p $ in $ L^2(Q_T) $, $ \mu^N \rightharpoonup \mu$ in $L^2(Q_T) $ and $ \Psi_{0}'(\p^N) \rightharpoonup \Psi_{0}'(\p) $ in $ L^2(Q_T) $ as $ N \to \infty $, we have 
\begin{align}\nonumber
& \lim_{N \to \infty} \left\{ \int_{0}^{\infty} \E(\p^N(t), \p^N(t))\,dt + h \int_{0}^{\infty}
\int_{\Omega} \nabla \p^N\cdot  \nabla \p^N\,dx\,dt \right\} \\\label{limitnorm}
& = \int_{0}^{\infty} \int_{\Omega} (\mu \p + \kappa \p^2)\,dx\,dt - \int_{0}^{\infty} \int_{\Omega} \Psi_{0}'(\p)\p\,dx\,dt 
 = \int_{0}^{\infty} \E(\p(t), \p(t))\,dt
\end{align}
because of \eqref{weakline3}.

Next we show $ \p^N \rightharpoonup \p $ in $ L^2(0, T; H^{\frac{\alpha}{2}}_{(m)}) $
and $ \sqrt{h} \nabla \p^N \rightharpoonup 0 $ in $ L^2(0, T; L^2) $ as $ N \to \infty $ for any $ T>0 $. Let $T>0$ be arbitrarily fixed.
$ ( \p^N )_{N \in \N} $ is bounded in $ L^{\infty}(0, T; H^{\frac{\alpha}{2}}_{(m)}) $, hence also in $ L^2(0, T;   H^{\frac{\alpha}{2}}_{(m)}) $. Then there exists some $ \p' \in L^2(0, T; H^{\frac{\alpha}{2}}_{(m)}) $ such that $ \p^N \rightharpoonup \p' $ in $ L^2(0, T; H^{\frac{\alpha}{2}}_{(m)}) $.
Since $ \p^N \to \p $ in $ L^2(Q_T) $, $ \p = \p' $. Hence $\p^N \rightharpoonup \p$ in $L^2(0, T; H^{\frac{\alpha}{2}}_{(m)}) $. 

For any fixed $ \bpsi \in C_0^{\infty}(Q_T)^d $, 
$$ \int_{Q_T} \sqrt{h}~\nabla \p_N \cdot \bpsi \, d(x, t) = - \int_{Q_T} \sqrt{h}~ \p^N~\mathrm{div}~\bpsi \, d(x, t) $$
 tends to zero as $N \to \infty$ since
$ \p^N \to \p $ in $ L^2(Q_T) $. 
Since $ \sup_{N \in \N} \| \sqrt{h} \nabla \p^N \|_{L^2(Q_T)^d} < \infty $ and
$ \overline{C_0^{\infty}(Q_T)^d}^{~\|\cdot\|_{L^2(Q_T)^d}} = L^2(Q_T)^d $, we have 
$\sqrt{h} \nabla \p^N \rightharpoonup 0$ in $L^2(Q_T)^d$. Hence we have
$ (\p^N, \sqrt{h} \nabla \p^N) \rightharpoonup (\p, 0) $ in $ L^2(0, T; H^{\frac{\alpha}{2}}_{(m)} \times (L^2)^d) $. 

Because of \eqref{limitnorm}, we also have the convergence of the norms of $ (\p^N, \sqrt{h} \nabla \p^N) $ to that of $ (\p, 0)$ in $ L^2(0, T; H^{\frac{\alpha}{2}}_{(m)} \times (L^2)^d) $. Hence we have shown the claim.

%Finally we can finish the proof by showing the energy inequality \eqref{weakline5}.
%Since $\bv^N(t) \to \bv(t)$ in $L^2(\Omega)^d$ and $\varphi^N(t) \to \varphi(t)$ in $H^1(\Omega)$ for
%almost every $t \in (0,\infty)$ (for a subsequence), which follows from the strong convergences of 
%$\bv^N$ and $\varphi^N$, it holds that
%\begin{align*}
% E^N(t) \to E_{\mbox{\footnotesize tot}}(\varphi(t),\bv(t)) \; \mbox{ for almost all } \, t \in (0,\infty) \,.
%\end{align*}
%Moreover, by lower semicontinuity of norms and almost everywhere convergence of $\varphi^N$ to $\varphi$,
%the inequality 
%\begin{align*}
% \liminf_{N \to \infty} \int_0^\infty D^N(t) \tau(t) \, dt \geq \int_0^\infty D(t) \tau(t) \, dt 
%\end{align*}
%for all $\tau \in W^{1,1}(0,\infty)$ with $\tau \geq 0$ holds, where
%\begin{align*}
% D(t) := \int_\Omega 2 \eta(\varphi) |D\bv|^2 \, dx + \int_\Omega m(\varphi) |\nabla \mu|^2 \, dx \,.
%\end{align*}
%Hence, passing to the limit in \eqref{integratedinequEandD}, we obtain
%\begin{align} \label{lastinequalityfinally}
% E_{\mbox{\footnotesize tot}}(\varphi_0,\bv_0) \tau(0) 
%   + \int_0^\infty E_{\mbox{\footnotesize tot}}(\varphi(t),\bv(t)) \, \tau(t) \, dt 
%  \geq \int_0^\infty D(t) \tau(t) \, dt 
%\end{align}
%for all $\tau \in W^{1,1}(0,\infty)$ with $\tau \geq 0$. With the help of Lemma \ref{lem:EnergyEstim} 
%we obtain the energy estimate~\eqref{weakline5}. 
%\hfill $\Box$ 

\section*{Acknowledgments} 
The results of this contribution were mainly obtained during a research stay of the second author at the University of Regensburg, which was partly supported by the ``Universit\"atsstiftung Hans Vielberth''. 
The second author would like to thank Professor Mitsuru Sugimoto for offering him the support by JSPS KAKENHI Grant Numbers 26287022.
The second author was also supported by JSPS KAKENHI Grant Numbers 17K17804.
These supports are gratefully acknowledged.

\end{document}